\documentclass[11pt]{amsart}
\usepackage{amssymb,amsmath,amsthm,amscd}
\usepackage{hyperref}
\usepackage{graphicx}
\usepackage{xcolor}
\numberwithin{equation}{section}

\newtheoremstyle{fancy1}{10pt}{10pt}{\itshape}{12pt}{\textsc\bgroup}{.\egroup}{8pt}{ } 
\newtheoremstyle{fancy2}{10pt}{10pt}{}{12pt}{\itshape}{.}{8pt}{ }

\theoremstyle{fancy1}

\newtheorem{lem}[equation]{Lemma} 
\newtheorem{prop}[equation]{Proposition}
\newtheorem{thm}[equation]{Theorem}
\newtheorem*{thm*}{Theorem} 

\newtheorem{main}{Theorem}
\newtheorem*{main*}{Theorem}

\newtheorem*{cor*}{Corollary}
\newtheorem*{prop*}{Proposition}

\newtheorem*{problem*}{Problem}

\theoremstyle{fancy2}

\newtheorem{rem}[equation]{Remark}
\newtheorem*{rems*}{Remarks}

\newtheorem*{rem*}{Remark}
\newtheorem{example}{Example}
\newtheorem*{example*}{Example}

\newcommand{\cref}[1]{Corollary~\ref{#1}}

\newcommand{\eref}[1]{Example~\ref{#1}}
\newcommand{\lref}[1]{Lemma~\ref{#1}}
\newcommand{\pref}[1]{Proposition~\ref{#1}}
\newcommand{\rref}[1]{Remark~\ref{#1}}
\newcommand{\tref}[1]{Theorem~\ref{#1}}
\newcommand{\sref}[1]{Section~\ref{#1}}

\newcommand{\fref}[1]{Figure~\ref{#1}}

\newcommand{\e}{\epsilon}

\newcommand{\M}{\mathcal M}
\newcommand{\DD}{\mathcal D}

\newcommand{\R}{{\mathbb{R}}}

\newcommand{\N}{{\mathbb{N}}}

\newcommand{\fg}{{\mathfrak{g}}}
\newcommand{\fk}{{\mathfrak{k}}}
\newcommand{\fh}{{\mathfrak{h}}}
\newcommand{\fm}{{\mathfrak{m}}}
\newcommand{\fn}{{\mathfrak{n}}}

\newcommand{\fl}{{\mathfrak{l}}}

\newcommand{\fz}{{\mathfrak{z}}}

\newcommand{\pro}[2]{\langle #1 , #2 \rangle}

\def\con#1=#2(#3){#1 \equiv #2 \bmod{#3}}

\newcommand{\ml}{\langle}                    
\newcommand{\mr}{\rangle}

\newcommand{\tr}{\ensuremath{\operatorname{tr}}}
\newcommand{\diag}{\ensuremath{\operatorname{diag}}}
\newcommand{\Aut}{\ensuremath{\operatorname{Aut}}}

\newcommand{\rank}{\ensuremath{\operatorname{rk}}}

\newcommand{\Ad}{\ensuremath{\operatorname{Ad}}}

\newcommand{\Ric}{\ensuremath{\operatorname{Ric}}}
\newcommand{\ric}{\ensuremath{\operatorname{ric}}}

\newcommand{\Hess}{\ensuremath{\operatorname{Hess}}}
\newcommand{\Inn}{\ensuremath{\operatorname{Inn}}}

\DeclareMathOperator{\grad}{grad}
\DeclareMathOperator{\Div}{div}

\begin{document}

\title{Palais--Smale sequences for the prescribed Ricci curvature functional}

\author{Artem Pulemotov}
\address{University of Queensland}
\email{a.pulemotov@uq.edu.au}
\author{Wolfgang Ziller}
\address{University of Pennsylvania}
\email{wziller@math.upenn.edu}

\thanks{This research was supported by the Australian Government through the Australian Research Council's Discovery Projects funding scheme (project DP220102530).}

\begin{abstract}
We obtain a complete description of divergent Palais--Smale sequences for the prescribed Ricci curvature functional on compact homogeneous spaces. As an application, we prove the existence of saddle points on generalized Wallach spaces and several types of generalized flag manifolds. We also describe the image of the Ricci map in some of our examples.
\end{abstract}

\maketitle

\bigskip

The prescribed Ricci curvature problem consists in finding a Riemannian metric $g$ on a manifold $M$ such that
\begin{align*}
\Ric(g)=cT
\end{align*}
for some constant~$c$. 
The inclusion of $c$ is required since the Ricci tensor is invariant under scaling. The prescribed Ricci curvature problem has been studied by many authors since the 1980s. While its local variant is relatively well understood (see \cite{D85}) general results for global solutions are scarce. It is thus natural to make symmetry assumptions. More precisely, suppose that the metric $g$ and the tensor $T$ are invariant under a compact Lie group $G$ acting on~$M$. In the case where the quotient $M/G$ is one-dimensional, the problem was addressed by several authors in a few special situations; see ~\cite{RH84,CD94,AP13a,BK20}. The case where $M$ is a homogeneous space~$G/H$ has been studied more extensively. For a detailed overview of the results, see the survey~\cite{BP19} and the more recent references~\cite{AGP21,APZ21,LWa,LWb,PZ22}. In some simple situations, the equation can be solved explicitly. In general, this direct approach is hopeless since the algebraic systems involved are too complicated.

As is well known \cite{AP16}, $G$-invariant metrics with Ricci curvature $cT$ are, up to scaling, critical points of the scalar curvature functional $S$ on the set $\M_T=\M_T(G/H)$ of $G$-invariant metrics on $G/H$ subject to the constraint $\tr_gT=1$. It is thus natural to study the problem from a calculus of variations viewpoint. We will assume throughout the paper that $T$ is positive definite; in fact, the behavior changes significantly otherwise. The positivity of $T$ implies that the functional is bounded from above, so the first question is to determine when it has a global maximum. In~\cite{PZ22} we proved a general existence theorem for global maxima. On the other hand, there are many homogeneous spaces where the functional does not assume its supremum. The next step is thus to search for saddle critical points, using a version of the mountain pass lemma. This approach has been very successful in the case of homogeneous Einstein metrics; see~\cite{BWZ04,CB04,BK21}.
 
Mountain pass techniques require some form of compactness. The assumption imposed most commonly is the Palais--Smale condition. It postulates that every sequence of metrics $g_i$ with 
  \begin{align*}
  	\lim_{i\to\infty} S(g_i)=\lambda\in\mathbb R \quad \text{ and }\quad \lim_{i\to\infty}|\grad S_{|\M_T}(g_i)|_{g_i}=0
  \end{align*}
has a convergent subsequence. One can view this as excluding the possibility of a critical point at infinity. In our case, the Palais--Smale condition can be reformulates as
  \begin{align*}
  	\lim_{i\to\infty} S(g_i)=\lambda \qquad \text{and} \qquad \lim_{i\to\infty}|\Ric(g_i)-\lambda T|_{g_i}=0.
  \end{align*}
\smallskip
However, in contrast to the case of Einstein metrics, which are also critical points of a functional, this condition is never satisfied unless $H$ is maximal in~$G$.

For an intermediate subgroup $K$ with $\dim H<\dim K<\dim G$, consider the fibration
\begin{align*}\label{homog_fibr}
K/H\to G/H\to G/K
\end{align*}
with fiber $F=K/H$ and base $B=G/K$. Given homogeneous metrics $g_F$ on $F$ and $g_B$ on $B$, we define the metric $g=g_F+g_B$ on $G/H$ making this projection into a Riemannian submersion. Scaling the fiber and the base, and normalizing so that $\tr_gT=1$, we obtain a one-parameter family of metrics $g_t\in \M_T$, called a {\it canonical variation}.
 
Our main result classifies divergent Palais--Smale sequences for the functional $S_{|\M_T}$.

\begin{main}\label{main_1} Let $G/H$ be a compact homogeneous space and
$(g_i)\subset\M_T$ a divergent Palais--Smale sequence of homogeneous metrics with $\lim_{i\to\infty} S(g_i)=\lambda$. Then there exist an intermediate subgroup $K$ and a subsequence $(g_{i_k})$ such that the homogeneous metric $g_F$ on $K/H$ given by
  $$g_F=\lim_{k\to\infty}g_{i_k|K/H}$$
 satisfies
	\begin{align*}
		\Ric(g_F)=\lambda T_{|K/H} \qquad \text{and} \qquad \tr_{g_F}T_{|K/H}=1.
	\end{align*}
Conversely, given an intermediate subgroup $K$, every solution $g_F$ to the first of these equations can be extended, possibly after re-scaling, to a metric on $G/H$ whose canonical variation gives rise to a divergent Palais--Smale sequence $(g_i)$ with $\lim_{i\to\infty} S(g_i)=\lambda$.
\end{main}

In other words, divergent Palais--Smale sequences correspond to critical points of the scalar curvature functional on the space $\M_T(K/H)$ of metrics on the fiber. Notice also that $\lambda=S(g_F)$, and  we call $\lambda$ a \emph{critical level at infinity}.

If $\lambda=0$, then $(g_i)$ is called a $0$--Palais--Smale sequence. Since a Ricci-flat homogeneous space must be flat, \tref{main_1} implies the following result.

\begin{cor*}\label{main_0PS}
If $S_{|\M_T}$ admits a $0$--Palais--Smale sequence, then there exists an intermediate subgroup $K$ with $K/H$ a torus.
\end{cor*}

 Given an intermediate subgroup $K$ with Lie algebra~$\fk$, we define two invariants:
\begin{align*}
	\alpha_\fk=\sup\Big\{\frac{S(h)}{\tr_hT_{|K/H}}\,\Big|\, h\in \M_T(K/H)\Big\}\qquad\text{and}\qquad  \beta_\fk=\sup\Big\{\frac{S(h)}{\tr_hT_{|G/K}} \,\Big|\, h\in \M_T(G/K) \Big\}.
\end{align*}
We note that there may not exist a metric on $K/H$ or $G/K$ that realizes the supremum.  Nevertheless, as was observed in~\cite{PZ22},  these quantities control the behavior of  $S_{|\M_T}$. Indeed, $\alpha_\fk$ is  one of the critical levels at infinity, and $\beta_\fk-\alpha_\fk$ can be interpreted as the derivative at infinity for a canonical variation.

In~\cite{PZ22} we showed that  the functional $S_{|\M_T}$ realizes its supremum  if the ``derivative" at the highest level is positive. The next natural goal is to find conditions under which this functional has a saddle point. The main difficulty is that $S_{|\M_T}$ does not satisfy the Palais--Smale condition. Our strategy for finding saddle points is thus as follows. Let $(a,b)$ be an interval containing no critical levels at infinity. Theorem~\ref{main_1} shows that, on the pre-image $S^{-1}((a,b))$ intersected with~$\M_T$, one can try to use a standard mounting pass argument. For this one assumes that there exists a value $c\in(a,b)$ such that $S^{-1}((c,\infty))\cap\M_T$ has two connected components. Choosing a curve between these components,
one uses the gradient flow to ``push" this curve up. Since it cannot pass level~$c$, this curve ``remains hanging" at a critical point of co-index 0 or~1.

One of the major challenges is to find a curve between the connected components of $S^{-1}((c,\infty))\cap\M_T$ that lies in $S^{-1}((a,\infty))\cap\M_T$. There are two methods we use to overcome this challenge on two particular types of spaces. The first is to construct the curve by combining two canonical variations. We show that this indeed works for the so--called generalized Wallach spaces, a large class of homogeneous spaces studied frequently in the recent literature. Some typical examples are $U(p+q+r)/U(p)U(q)U(r)$ and their analogues for the orthogonal and symplectic groups, as well the Ledger--Obata spaces $H^4/\diag H$ with $H$ simple. For a classification of generalized Wallach spaces, see~\cite{Ni07}. Each such space admits precisely three intermediate subgroups $K_i$. We have the following result.

\begin{main}\label{main_2}
	Let $G/H$ be a generalized Wallach space with  three  intermediate subgroups $K_i$, $i=1,2,3$. If two of the quantities $\beta_{\fk_i}-\alpha_{\fk_i}$ are negative, then the functional $S_{|\M_T}$ has a critical point with co-index $0$ or~$1$.
\end{main}
 
The conditions of the theorem imply that there exists an ordering such that $\alpha_{\fk_1}<\alpha_{\fk_2}\le\alpha_{\fk_3}$,
$\beta_{\fk_2}-\alpha_{\fk_2}<0$ and $\beta_{\fk_3}-\alpha_{\fk_3}<0$. 
If we let $a=\alpha_{\fk_1}$ and $b=\alpha_{\fk_2}$, then the scalar curvature functional satisfies the Palais--Smale condition on $S^{-1}((a,b))\cap\M_T$ by \tref{main_1}.
 We then prove that there exists $\e>0$ such that $\{g\in\M_T(G/H)\mid S(g)>b-\e\}$ has two components. This involves estimating $S_{|\M_T}$ near infinity at level $b$ and exploiting the assumption that the ``derivative" is negative.  A computation shows that the composition of two canonical variations connects these components and lies in  $S^{-1}((a,b))$.

The second method we use is to choose $a=-\infty$ and $b=\alpha_\fk$ for a subgroup $K$ such that~$\alpha_\fk$ has the lowest possible value. If there exists a second intermediate subgroup $K'$, we show that $S^{-1}((c,\infty))\cap\M_T$ has two connected components for some $c<\alpha_\fk$ provided $\beta_{\fk}-\alpha_{\fk}<0$. Moreover, we can  find  a curve connecting these components. However, we encounter here another problem: for some intermediate subgroup~$K''$, the scalar curvature functional on $\M_T(K''/H)$ may have critical points with scalar curvature below~$\alpha_\fk$. According to \tref{main_1}, these critical points give rise to divergent Palais--Smale sequences. While the invariants $\alpha_{\fk}$ are easily computable in many situations, there is no general theory for finding the scalar curvature at the other critical levels at infinity. If the space $K/H$ is simple enough  (e.g., it has at most two irreducible summands in its isotropy representation), then the functional on the fiber has only one critical point, allowing us to overcome this problem. We illustrate this in the case of generalized flag manifolds with three or four isotropy summands, classified in~\cite{Ki90,AC10}. Recall that  $G/H$ is a generalized flag manifold if $H$ is the centralizer of a torus in~$G$. On homogeneous spaces with one summand, the prescribed Ricci curvature problem is trivial, and on those with two summands, it was resolved in~\cite{AP16}. We will show:

\begin{main}\label{main_3}
Let $G/H$ be a generalized flag manifold with three or four isotropy summands and $\fk$ an intermediate subalgebra such that $\alpha_\fk<\alpha_{\fk'}$ for every other intermediate subalgebra~$\fk'$. If $\beta_\fk-\alpha_\fk<0$, then the functional $S_{|\M_T}$ has a critical point of co-index 0 or~1.
\end{main}

The condition that the generalized flag manifold have at most four isotropy summands guarantees that the scalar curvature functional on $\M_T(K'/H)$ has no critical values below~$\alpha_\fk$. 
\smallskip

We illustrate our results in several examples.

\begin{itemize}
	\item
	For the Wallach space $SU(3)/\mathbb T^2$ we draw the region where $S_{|\M_T}$ has a global maximum or a saddle point in Figure \ref{Wallach_regions}. The graphs of $S_{|\M_T}$ for some typical choices of $T$ are shown in Figure~\ref{Wallach_critical}. A computer experiment with two million data points indicates that Theorem~\ref{main_2} is optimal in this example.
	\item
We examine the bifurcations of critical points of $S_{|\M_T}$ on $G_2/U(2)$ as $T$ varies in the space of positive-definite bi-linear forms. One can start with $T$ where the functional has a unique critical point, namely a global maximum, and move to $T$ where,  in addition to the  global maximum, a saddle point appears. As one continues, the global maximum changes to a local maximum, with the supremum of $S_{|\M_T}$ not achieved, and the saddle remains. Finally, these two critical points merge into one degenerate saddle point, which then disappears. We depict graphs illustrating these bifurcations  in \fref{fig_G2_cr_pts}. 
\item  We also discuss the image of the Ricci map. We will show that it is the union of finitely many smooth immersed hypersurfaces meeting at the image of the semi-algebraic set in $\M$ of degenerate critical points; cf.~\pref{Morse} and see Figures~\ref{Wallach-Ric-image} and~\ref{fig_sing_curves} for two examples.
\item
For a generalized flag manifold with four isotropy summands, we show the regions where Theorem~\ref{main_3} applies and observe how they change as different subgroups realize the minimum among the values of~$\alpha_\fk$.
\end{itemize}

We will also show that there exists a  subset $A$ of full measure in the space of homogeneous metrics such that $S_{|\M_T}$ is a Morse function for any $T\in A$; see \pref{Morse}.
\smallskip

It is interesting to compare the behavior of $S_{|\M_T}$ to that of $S_{|\M_1}$, where $\M_1$ is the set of homogeneous metrics of volume 1. The critical points of the latter functional are homogeneous Einstein metrics, and they have been studied extensively. There are similarities between $S_{|\M_T}$ and $S_{|\M_1}$ but also important differences. Unlike $S_{|\M_T}$, the functional $S_{|\M_1}$ is bounded above only if $H$ is maximal in $G$. When this is not the case, the behavior of $S_{|\M_1}$ at infinity is again controlled by the collection of intermediate subgroups. A graph theorem, as well as the topology of the simplicial complex of intermediate subgroups, guarantee the existence of a critical point of co-index 0 or 1 in a large class of examples; see~\cite{BWZ04,CB04}. Furthermore, unlike $S_{|\M_T}$, the functional $S_{|\M_1}$ satisfies the Palais--Smale condition on the set of metrics with positive scalar curvature. As a consequence, the set of critical points has only finitely many components, each of them compact. On the other hand, critical submanifolds of $S_{|\M_T}$ (if they exist) are typically non-compact. It is natural to ask whether one can obtain an analogue of the graph theorem in~\cite{BWZ04} in the context of the prescribed Ricci curvature problem on general homogeneous spaces. The presence of divergent Palais--Smale sequences is an obvious obstacle. Nevertheless, as Theorems~\ref{main_2} and~\ref{main_3} show, this obstacle can be circumvented in some situations.

Many of the difficulties in proving our results stem from the fact that $S_{|\M_T}$ does not admit a continuous extension to the ``boundary" of~$\M_T$ and thus $\beta_\fk-\alpha_\fk$ is not an actual derivative. Because of this, we need to carry out careful estimates for $S_{|\M_T}$ at infinity. The above strategies for producing saddle points clearly can be effective on other types of homogeneous spaces, not just on generalised Wallach spaces or generalised flag manifolds. We  chose these two large classes of spaces for illustration.

The paper is organized as follows. In \sref{prelim} we recall properties of homogeneous spaces and some of the results from~\cite{PZ22} that we will need for our proofs. In particular, we discuss the simplicial complex that compactifies the set of metrics~$\M_T$, the stratification of its boundary, and the behavior of the scalar curvature near the strata. In \sref{sec_PS} we classify divergent Palais--Smale sequences and prove \tref{main_1}. In \sref{Wallach_saddle} we discuss generalized Wallach spaces and in \sref{sec_graph} generalized flag manifolds. Finally, \sref{sec_Examples} focuses on examples.
 
\section{Preliminaries}\label{prelim}

In this section, we briefly summarize the background we need. For a more detailed exposition, see~\cite{PZ22}.  
We  assume that $G$ and $H$ are compact  and that $G/H$ is not covered by a torus (otherwise all the $G$-invariant metrics are flat). We also assume for simplicity that 
$G$ and $H$ are connected, and that all the intermediate subgroups are connected and compact. Our results still hold without connectedness and without the assumptions on the intermediate subgroups; see~\cite[Section~1]{PZ22} for an explanation.

Let~$\fh\subset\fg$ be the Lie algebras of $H\subset G$. Assume that $G/H$ is an almost effective homogeneous space. We fix a bi-invariant metric $Q$ on~$\fg$. It defines a $Q$-orthogonal $\Ad_H$-invariant splitting $\fg=\fh\oplus\fm$ with $\fm=\fh^\perp$. The tangent space $T_{eH} (G/H)$ is identified with $\fm$, and $H$ acts on $\fm$ via the adjoint representation $\Ad_H$.  A $G$-invariant metric on $G/H$ is determined by an $\Ad_H$-invariant inner product on~$\fm$. We denote by $\M(G/H)$, or sometimes simply by~$\M$, the space of $G$-invariant metrics on $G/H$.

Consider a $Q$-orthogonal decomposition
\begin{equation*}
\fm=\fm_1\oplus\ldots\oplus\fm_r,
\end{equation*}
where $\Ad_H$ acts irreducibly on $\fm_i$. Some of the summands $\fm_i$ may need to  be one-dimensional if there exists a subspace of $\fm$ on which $\Ad_H$ acts as the identity. We denote by $\DD$ the space of all such decompositions and use the letter $D\in\DD$ for a particular choice of decomposition. The space $\DD$ has a natural topology induced from the embedding into the product of Grassmannians $G_k(\fg)$ of $k$-dimensional subspaces of~$\fg$. Clearly, $\DD$ is compact.

If $T$ is a $G$-invariant symmetric bi-linear form field on~$G/H$, it is determined by its value on~$\fm$. We are interested when such a bi-linear form field is (up to scaling) the Ricci curvature of a metric~$g\in\M$, i.e., when
\begin{equation}\label{prescribe}
\Ric(g)=cT \qquad \text{for some constant $c$.}
\end{equation}
Throughout the paper we will assume that $T$ is positive definite. We may also assume that $c>0$ in~\eqref{prescribe} since a compact homogeneous space does not admit any metrics with $\Ric\le 0$ unless it is flat (see~\cite[Theorem~1.84]{AB87}), which we excluded above.

Define the hypersurface
\begin{equation*}
\M_T(G/H)=\{g\in\M\mid \tr_gT=1\}\subset\M,
\end{equation*}
where $\tr_g$ is the trace with respect to $g$. We denote it simply by $\M_T$ when the homogeneous space is clear from context. As shown in~\cite{AP16}, a solution to~\eqref{prescribe} can we viewed as a critical point of the functional
\begin{equation*}
S\colon \M_T \to \R,
\end{equation*}
where $S(g)$ is the scalar curvature of $g$. 

We now recall the formulas for the scalar curvature and the Ricci curvature of a homogeneous metric. Given $g\in\M$, we have  $g_{|\fm_i}=x_i\, Q_{|\fm_i}$ for some constant $x_i>0$. In general, $\fm_i$ and $\fm_j$ do not have to be $g$-orthogonal if some of these summands are equivalent. However, we can diagonalize $g$ and $Q$ simultaneously, and hence there exists a decomposition $D\in\DD$ such that the metric has the form 
\begin{equation}\label{diagonalmetric}
g=x_1Q_{|\fm_1}+ x_2Q_{|\fm_2} +\cdots+  x_r Q_{|\fm_r}.
\end{equation}
We call such metrics diagonal with respect to our choice of $D$  
and denote their set by $\M^D(G/H)$, or simply $\M^D$. Thus, $\M=\cup_{D\in\DD}\M^D$. We also denote $\M_T^D=\M_T\cap\M^D$. The scalars $x_i$ are simply the eigenvalues of $g$ with respect to~$Q$.

We define the structure constants
\begin{equation*}
[ijk]=[ijk]^D=\sum_{\alpha,\beta,\gamma} Q([e_\alpha,e_\beta],e_\gamma)^2,\qquad i,j,k=1,\cdots,r,
\end{equation*}
where $(e_\alpha)$, $(e_\beta)$ and $(e_\gamma)$ are $Q$-orthonormal bases of $\fm_i$, $\fm_j$ and~$\fm_k$. Clearly, $[ijk]\ge 0$, and $[ijk]=0$ if and only if $Q([\fm_i,\fm_j],\fm_k)=0$. We will denote by $B$  the Killing form of~$G$. By the irreducibility of~$\fm_i$, there exist constants $b_i=b_i^D\ge0$ such that 
\begin{equation*}
B_{|\fm_i}=-b_i Q_{|\fm_i} 
\end{equation*}
with $b_i=0$ if and only if $\fm_i$ lies in the center of $\fg$. Furthermore,  not all of $b_i$ vanish since otherwise $\fm$ is in the center $\fz(\fg)$ and $G/H$ is flat. Clearly, the numbers $b_i$ and the structure constants $[ijk]$ depend continuously on~$D$.

Using the notation $d_i=\dim\fm_i$, the scalar curvature is given by
\begin{align}\label{scalcurvx}
S(g)&=\frac12\sum_{i}\frac{d_ib_i}{x_i}-\frac14\sum_{i,j,k}[ijk]\frac{x_k}{x_ix_j}
\end{align}
(see~\cite{MWWZ86}), and the Ricci curvature satisfies
\begin{align}\label{RiccDiag}
\Ric(g)_{|\fm_i}&=\bigg(\frac{b_{i}}2
-\frac1{2d_i} \sum_{j,k} [ijk]\frac{x_{k}}{x_{j}}
+\frac1{4d_i}\sum_{j,k}[ijk]\frac{ x_{i}^2}{x_{j}x_{k}  }\bigg)Q_{|\fm_i}.
\end{align}
For a diagonal metric, the Ricci curvature is not necessarily diagonal, and the off-diagonal terms are given by
\begin{align}\label{RiccOffDiag}
\Ric(g)(u,v)&= \sum_{k,l}\Big(\frac{x_ix_j-2 x_k^2+2x_kx_l}{4x_kx_l}\Big)
\sum_{e_\alpha\in\fm_k}Q\big([u,e_\alpha]_{\fm_l},[v,e_\alpha]_{\fm_l}\big),\qquad u\in \fm_i,~v\in \fm_j,
\end{align}
where $i\ne j$ and the subscript $\fm_l$ denotes projection onto $\fm_l$; see \cite{GZ02}.

For the tensor $T$, we introduce the constants $T_i=T_i^D$ such that
\begin{equation*}
T_{|\fm_i}=T_i\, Q_{|\fm_i}.
\end{equation*}
Varying over all decompositions, these constants determine~$T$ uniquely.

It will be convenient for us to describe a homogeneous metric in terms of its inverse since this makes the space of metrics subject to our constraint pre-compact. We will use the following parametrization of~$\M$. Given $y=(y_1,\ldots,y_r)\in\mathbb R_+^r$ and a decomposition $D\in\mathcal D$ with irreducible modules $\mathfrak m_1,\ldots,\mathfrak m_r$, consider the metric
\begin{align}\label{metr_terms_inv}
g=\sum_i\tfrac1{y_i}Q_{|\mathfrak m_i}.
\end{align}
In what follows, we identify $g$ with $(y,D)=((y_1,\ldots,y_r),D)$. Notice that  $g$ may be diagonal with respect to multiple decompositions. Thus, the identification of $g$ with $(y,D)$ is not one-to-one.

If $g\in\M_T^D$ is given by~\eqref{metr_terms_inv}, we obtain the following formulas for the scalar curvature and our constraint:
\begin{equation}\label{scalar_y}
S(g)=\frac12\sum_{i}d_ib_iy_i-\frac14\sum_{i,j,k}\frac{y_iy_j}{y_k}[ijk],
\qquad \tr_gT=\sum_{i}d_iT_iy_i=1.
\end{equation}
We need to study the behavior of $S_{|\M_T^D}$ at infinity, which means that at least one of the variables $y_i$ goes to $0$. It is natural to introduce a simplicial complex and its stratification. Specifically, let 
\begin{equation*}
\Delta=\Delta^D=\big\{(y_1,\cdots, y_r)\in \R^r\,\big|\, \textstyle\sum_{i} d_iT_iy_i=1 \text{ and } y_i> 0\big\},
\end{equation*}
 which is a natural parametrization of the set~$\M_T^D$.   We identify a metric $g\in\M_T^D$ with $y=(y_1,\cdots, y_r)\in\Delta$ when the choice of the decomposition $D$ is understood from the context. We emphasize that the numbers $x_i$, $T_i$, and  the simplex~$\Delta$, depend on the choice of~$D$.

The boundary of $\Delta$ consists of lower-dimensional simplices. For every nonempty proper subset $J$ of the index set $I=\{1,\ldots,r\}$, let $J^c=I\setminus J$ and
\begin{equation*}
\Delta_J=\{y\in\partial\Delta\mid y_i>0 \text{ for } i\in J,\ y_i=0 \text{  for } i\in J^c\}.
\end{equation*}
Thus $\Delta_J$ is a $|J|$-dimensional simplex, which we call a \emph{stratum} of~$\partial\Delta$.  
The closure of $\Delta_J$ satisfies
\begin{equation*}
\bar{\Delta}_J=\bigcup_{J'\subset J}\Delta_{J'},
\end{equation*}
and we call $\Delta_{J'}$ a stratum adjacent to $\Delta_{J}$ if $J'$ is a nonempty proper subset of~$J$. It will also be useful for us to consider tubular $\e$-neighborhoods of strata for $\e>0$:
\begin{equation*}
T_\e(\Delta_J)=\{y\in\Delta \mid y_i\le \e \text{ for }  i\in J^c\}.
\end{equation*}
Finally, we associate to each stratum an $\Ad_H$-invariant subspace of $\fm$:
\begin{equation*}
\fm_J=\bigoplus_{i\in J}\fm_i.
\end{equation*}
We can add markings to the strata in $\partial\Delta$. If $\fh\oplus\fm_J=\fk$ is a subalgebra, we denote the stratum $\Delta_J$ by $(\Delta_J, \fk)$ or simply $\Delta_\fk$; if it is not, we denote the stratum by~$(\Delta_J,\infty)$.

 We will also use the formula for a Riemannian submersion with totally geodesic fibers. If $g_F$ is a metric on the fiber $F$ and $g_B$ a metric on the base $B$, consider the two-parameter family of metrics
\begin{equation*}
	g_{s,t}=\frac1s \,g_F + \frac1t\,g_B.
\end{equation*}
We focus on the case where $g_B$ and $g_F$ are homogeneous metrics on $B=G/K$ and $F=K/H$ for an intermediate subgroup~$K$. As follows from ~\cite[Proposition~9.70]{AB87}, the Ricci curvature of $g_{s,t}$ is given by
\begin{align}\label{SubRicci}
	\Ric_{s,t}(u,v)=
	\begin{cases}
		\Ric_F(u,v)+\frac{t^2}{s^2}\ml Au,Av\mr_g & \text{ for } u,v\in \fk\cap\fm, \\  
		\Ric_B(u,v)-2\frac{t}{s}\ml A_u,A_v\mr_g & \text{ for } u,v\in \fk^\perp, \\
		\frac{t}{s}\Div A\,(u,v) & \text{ for } u\in \fk\cap\fm,~v\in\fk^\perp,
	\end{cases}
\end{align}
where $A$ is the O'Neil tensor of the fibration, and the scalar curvature is 
\begin{align}\label{SubScal}
	S(g_{s,t})= s S_F + t S_B-\frac{t^2}{s}|A|_g
\end{align}
with $S_F$ and $S_B$ the scalar curvatures of $g_F$ and~$g_B$.  Notice  that, if $K$ is not closed (and hence $K/H$ may not be Hausdorff), formulas \eqref{SubRicci} and~\eqref{SubScal} still hold on the level of Lie algebras since they are local in nature.  Solving  the constraint $\tr_{g_{s,t}}T=1$ for~$s$, we obtain a family $(g_t)\subset S_{|\M_T} $. The scalar curvatures of $g_t$ is
\begin{equation}\label{scal_g}
	S(g_t)= \frac{S_F}{T_1^*}+T_2^*\Big(\frac{S_B}{T_2^*}-\frac{S_F}{T_1^*}\Big)t-\frac{t^2T_1^*}{1-tT_2^*}|A|_g,
\end{equation}
where $T_1^*=\tr_{g_F}T_{|F}$ and $T_2^*=\tr_{g_B}T_{|B}$ are the traces of $T$ on the fiber and the base. As a consequence,
\begin{equation}\label{derivative}
	\lim_{t\to0}S(g_t)=\frac{S_F}{T_1^*} \qquad\text{and}\qquad
	\lim_{t\to0}\frac{dS(g_t)}{d t}=T_2^*\Big(\frac{S_B}{T_2^*}-\frac{S_F}{T_1^*}\Big).
\end{equation}
The family $(g_t)$ is called the \emph{canonical variation} of~$g=g_F+g_B$.

Motivated by \eqref{derivative},  we define the following invariants for every intermediate subgroup $K$ with Lie algebra~$\fk$:
\begin{align*}
	\alpha_\fk&=\sup \{S(g)\mid g\in \M(K/H)\text{ with }\tr_gT_{|\fk\cap\fm}=1\}, \\
 \beta_\fk&=\sup \{S(g)\mid g\in \M(G/K)\text{ with }\tr_gT_{|\fk^\perp}=1\}.
\end{align*}
It is important to note that the suprema in these definitions are not always achieved. However, there exist an intermediate subgroup $K'\subset K$ (with Lie algebra $\fk'$) and a metric $g\in \M(K'/H)$ such that $\alpha_\fk=\alpha_{\fk'}$ and $S(g)=\alpha_\fk'$.
We also recall that
\begin{equation}\label{inclusion}
\text{ if }\quad \fk_1\subset\fk_2, \quad \text{then}\quad \alpha_{\fk_1} \le \alpha_{\fk_2}.
\end{equation}
For the proofs of these properties, see~\cite{PZ22}. Furthermore, the following result is an immediate consequence of~\cite[Propositions~4.1 and~4.2]{PZ22}. 

\begin{prop}\label{prop_max_est}\label{limit_subalgebra}
Given $D\in\mathcal D$ and $\e>0$, if $y$ lies in a stratum $(\Delta_J^D,\infty)$, then there exists a neighborhood $U$ of $(y,D)$ in $\mathbb R^r\times \mathcal D$ such that $S(g)<-\e$ for every metric $g=(y',D')$ contained in~$U$. If $y$ is in a stratum $\Delta_\fk^D$, then there is a neighborhood $V$ of $(y,D)$ in $\mathbb R^r\times \mathcal D$ such that $S(g)<\alpha_\fk+\e$ for every $g=(y',D')$ in~$V$.
\end{prop}

Among other things, this proposition implies that, if $g_i\in\M_T$  is a divergent sequence with $S(g_i)$ bounded, there exist a decomposition $D$ and an intermediate subalgebra $\fk$ such that some subsequence of $g_i$ has a limit in $\Delta_\fk^D$.

We finally state the main theorem in~\cite{PZ22}. This result is a sufficient condition for the existence of a maximum.

\begin{thm}{\label{Ex_max}}	If $H$ is not maximal in $G$, then the set of intermediate subgroups $K$ with $\alpha_\fk$ equal to $\sup_\fl\alpha_\fl$ is non-empty. If $K$ is such a subgroup of the lowest possible dimension and ${\beta_{\fk}-\alpha_{\fk}>0}$, then $S_{|\M_T}$ achieves its supremum.
\end{thm}

\section{Palais--Smale sequences}\label{sec_PS}

In order to find critical points that are not global maxima via mountain pass techniques, one needs certain compactness properties for the functional~$S_{|\M_T}$. Normally, this is achieved by verifying the Palais--Smale condition. However, \pref{PS_canonical} below implies that  this condition cannot be satisfied for $S_{|\M_T}$ unless $H$ is maximal in~$G$. Even so, understanding the properties of divergent Palais--Smale sequences in~$\M_T$ is an important ingredient in the variational analysis of~$S_{|\M_T}$.

A sequence of metrics $(g_i)\subset\M_T$ is called a \emph{Palais--Smale sequence} for $S_{|\M_T}$ if $S(g_i)$ is bounded and
\begin{align*}
\lim_{i\to\infty}|\grad S_{|\M_T}(g_i)|_{g_i}=0,
\end{align*}
where the gradient and the norm are taken with respect to the metric on the tensor bundle of $\M$ induced by~$g_i$. The Palais--Smale condition is satisfied if every such sequence has a convergent subsequence. 
Note that we may also  assume, by going to a subsequence, that $S(g_i)$ converges to some $\lambda\in\mathbb R$ as $i\to\infty$.

To clarify this condition, we start by computing $\grad S_{|\M_T}$. It is well known (see, e.g.,~\cite[Proposition~4.17]{AB87}) that the gradient of the functional $S:\M\to \R$ satisfies
$$
\grad S(g)=-\Ric(g).
$$
The tangent space of $\M_T$ at the metric $g$ consists of those (0,2)-tensor fields $h$ on $G/H$ for which $\ml h,T\mr_g=0$. Projecting $\grad S$ onto this space, we find
\begin{equation}\label{gradS_compute}
\grad S_{|\M_T}(g)=-\Ric(g)+\frac{\ml \Ric(g),T\mr_g}{\ml T,T\mr_g}T.
\end{equation}
Thus, as expected,  $\grad  S_{|\M_T}(g)=0$ if and only if $\Ric(g)=c T$ for some constant~$c$.

Let $(g_i)$ be a Palais--Smale sequence with $\lim_{i\to\infty} S(g_i)=\lambda$. Substituting $g_i$ into~\eqref{gradS_compute}, taking the trace with respect to $g_i$, and using $\tr_{g_i}T=1$, it follows that 
$$
\lim_{i\to\infty} \frac{\ml \Ric(g_i),T\mr_{g_i}}{\ml T,T\mr_{g_i}} =\lambda.
$$
Altogether, we see that $g_i$ a Palais--Smale sequence with $\lim_{i\to\infty} S(g_i)=\lambda$ if and only if
\begin{align}\label{PS_sequence}
\lim_{i\to\infty} S(g_i)=\lambda \qquad \text{and} \qquad \lim_{i\to\infty}|\Ric(g_i)-\lambda T|_{g_i}=0.
\end{align}

We first construct examples of divergent Palais-Smale sequences.  Recall that given an intermediate subgroup $K$, we consider the homogeneous fibration $K/H\to G/H\to G/K$. Fixing a homogeneous  metric $g_F$ on the fiber $F$ and $g_B$  on the base $B$,  we obtain a two-parameter family of metrics 
$g_{s,t}$ on the total space $G/H$: 
 \begin{equation*}
	g_{s,t}=\frac1s \,g_F + \frac1t\,g_B.
\end{equation*}
Solving  the constraint $\tr_{g_{s,t}}T=s\tr_{g_F}T_{|K/H}+t\tr_{g_B}T_{|G/K}=1$ for~$s$, we obtain the family $(g_t)\subset S_{|\M_T} $, which is called a \textit{canonical variation}.

\begin{prop}\label{PS_canonical}
Fix an intermediate subgroup~$K$ and a positive-definite tensor $T$, and assume that the homogeneous space $K/H$ supports a $K$-invariant metric $g_F$ such that
	\begin{align}\label{PRC_on_KH}
		\Ric(g_F)=\lambda T_{|K/H}\qquad\mbox{and}\qquad\tr_{g_F}T_{|K/H}=\tfrac12
	\end{align}
	for some $\lambda\ge0$. After choosing a $G$-invariant metric $g_B$ on~$G/K$ with $\tr_{g_B}T_{|G/K}=\tfrac12$, we obtain the canonical variation $g_t$.  Then the metrics $g_{1/i}$ form a divergent Palais--Smale sequence for the functional $S_{|\mathcal M_T}$ with $\lim_{i\to\infty} S(g_{1/i})=\lambda$.
\end{prop}

\begin{rem} \label{more-divergent} (a)
Not all divergent Palais--Smale sequences arise in this fashion. For instance, consider the curve of metrics $g(a)=(x_1(a),x_2(a),a)$ on the Wallach space (see \sref{Wallach_ex}), where $x_2(a)=\sqrt{ba+a^2}$ for some $b\ge0$ and $x_1(a)$ is such that $\tr_{g(a)}T=1$. One can easily check that $(g(i))$ is a divergent Palais--Smale sequence and that it comes from a canonical variation only if $b=0$. 
	
	(b) If $\fk$ is an intermediate subalgebra such that $\alpha_\fk$ is not assumed in $\M_T(K/H)$, we still have a canonical variation with limit $\alpha_\fk$ since there exists an intermediate subalgebra $\fk'\subset \fk$ with  $\alpha_{\fk'}=\alpha_\fk$ and a metric $g\in\M_T(K'/H)$ with $S(g)=\alpha_\fk$.    Hence all the values $\alpha_\fk$ are critical levels at infinity. 
\end{rem}
\begin{proof}
Taking the trace of the first equality in~\eqref{PRC_on_KH}, we find $S(g_F)=\frac\lambda2$. In light of~\eqref{derivative},
	$$
	\lim_{t\to0}S(g_t)=\frac{S(g_F)}{\tr_{g_F}T_{|K/H}}=\lambda.
	$$
	It remains to show that $|\grad S_{|\M_T}(g_t)|_{g_t}$ goes to~0. Denote by 
	$\Ric_t$ the Ricci curvature of~$g_t$. According to \eqref{SubRicci},
	\begin{align*}
		\Ric_t(u,v)=
		\begin{cases}
			\Ric_F(u,v)+\frac{t^2}{(2-t)^2}\ml Au,Av\mr_g &\text{for}~u,v\in \fk\cap\fm, 
			\\ 
			\Ric_B(u,v)-\frac{2t}{2-t}\ml A_u,A_v\mr_g &\text{for}~u,v\in \fk^\perp,
			\\
			\frac{t}{2-t}\,\Div A\,(u,v) &\text{for}~u\in\fk\cap\fm,\ v\in\fk^\perp.
		\end{cases}
	\end{align*}
	Since   $\lim_{t\to0}|\Ric_B|_{g_t}=0$, due to the fact that the metric on the base blows up, we get
	\begin{align*}
		\lim_{t\to0}|\Ric_t-\pi^*\Ric_F|_{g_t}=0,
	\end{align*}
	where $\pi$ is the $Q$-orthogonal projection from $\fm$ onto~$\fk$. This implies
	\begin{align*}
		\lim_{t\to0}|\grad S_{|\M_T}(g_t)|_{g_t}&=\lim_{t\to0}\bigg|\Ric_t-
		\frac{\ml\Ric_t,T\mr_{g_t}}{\ml T,T\mr_{g_t}}T\bigg|_{g_t}
		\\
		&=\bigg|\Ric_F-
		\frac{\ml\Ric_F,T_{|K/H}\mr_{g_F}}{\ml T_{|K/H},T_{|K/H}\mr_{g_F}}T_{|K/H}\bigg|_{g_F}=|\Ric_F-\lambda T_{|K/H}|_{g_F}=0,
	\end{align*}
	which shows that $(g_{1/i})$ is a divergent Palais--Smale sequence.
\end{proof}

\begin{rem}
One could easily misinterpret equation~\eqref{derivative} to say that the canonical variation gives rise to a Palais--Smale sequence if the ``derivative" $\beta_\fk-\alpha_\fk$ vanishes. In the pictures in the coordinates~$y_i$, that is indeed what it looks like. However, this is misleading. For example, in the coordinates~$x_i$, the corresponding first derivative is~$0$, and the quantity $\beta_\fk-\alpha_\fk$ is only related to the second derivative. The reason is that the equality $\grad S(g)=-\Ric(g)$ only holds if the gradient is taken with respect to the metric~$g$. The results in this section show that Palais--Smale sequences are characterised by their convergence to criticals points of~$S_{|\M_T(K/H)}$, which is unrelated to the ``derivative". Instead, these ``derivatives" are important for estimating the scalar curvature; cf.~\pref{neg_der} and \cite[Proposition 3.7]{PZ22}.
\end{rem}

We now  prove the converse of \pref{PS_canonical}, thus classifying all  divergent Palais--Smale sequences. 
 Let $(g_i)\subset\M_T$ be a divergent Palais--Smale sequence. For every $i\in\N$, there exists a decomposition $D_i\in\DD$ with modules $\mathfrak m_1^i,\ldots,\mathfrak m_r^i$ such that
\begin{align*}
g_i=\tfrac1{y_{1i}}Q_{|\fm_1^i}+\tfrac1{y_{2i}}Q_{|\fm_2^i}+\cdots+\tfrac1{y_{ri}}Q_{|\fm_r^i}
\end{align*}
for some $y_{1i},\ldots,y_{ri}>0$. As per the conventions set out in Section~\ref{prelim}, we identify $g_i$ with $((y_{1i},\ldots,y_{ri}),D_i)$. Passing to a subsequence if necessary, we may assume that $D_i$ converge to a decomposition $D$ with modules $\fm_1,\ldots,\fm_r$ and the points $(y_{1i},\ldots,y_{ri})$ converge to a point $(y_{1},\ldots,y_{r})$ in the closure of the simplex~$\Delta=\Delta^D$. Since $(g_i)$ diverges, $(y_1,\ldots,y_r)$ must lie in the boundary of~$\Delta$.

\begin{thm}\label{thm_PSnec}
Let $(g_i)$ be a divergent Palais--Smale sequence with $\lim_{i\to\infty} S(g_i)=\lambda$. There exists an intermediate subgroup $K$ with Lie algebra $\mathfrak k$ such that the limit point $(y_1,\ldots,y_r)$ lies in the stratum~$(\Delta_J,\fk)$. Furthermore, the metric $g_F$ on $K/H$ given by
\begin{align}\label{gF_def}
g_F=\sum_{j\in J}\tfrac1{y_j}Q_{|\mathfrak m_j}
\end{align}
satisfies 
\begin{align}\label{PS_nec}
\Ric(g_F)=\lambda T_{|K/H}. 
\end{align}
\end{thm}

\begin{proof}
 The existence of $K$ such that $(y_1,\ldots,y_r)$ lies in $(\Delta_J,\fk)$ follows from the convergence of $S(g_i)$ and Proposition~\ref{limit_subalgebra}. It remains to show that the metric $g_F$ given by~\eqref{gF_def} satisfies~\eqref{PS_nec}.

Denote by $b_{ji}$, $[jkl]_i$ and~$T_{ji}$ the constants associated with the decomposition~$D_i$ which converge to the constants  $b_j$, $[jkl]$ and $T_j$ associated to $D$. Without loss of generality, assume that the dimension $d_j=\dim\fm_j^i$ does not depend on~$i$. According to~\eqref{RiccDiag},
\begin{align}\label{RicciDiagonal}
\Ric(g_i)_{|\fm_j^i}&=\bigg(\frac{b_{ji}}2
-\frac1{2d_j} \sum_{k,l=1}^r [jkl]_i\frac{y_{li}}{y_{ki}}
+\frac1{4d_j}\sum_{k,l=1}^r[jkl]_i\frac{y_{ki}y_{li}}{y_{ji}^2}\bigg)Q_{|\fm_j^i}
\qquad \text{for all $j\in I$}.
\end{align} 
Thus \eqref{PS_sequence} implies that
\begin{align*}
\frac{b_jy_j}2-\frac1{2d_j}\lim_{i\to\infty}&\bigg(
\sum_{k,l=1}^r [jkl]_i\frac{y_{li}y_{ji}}{y_{ki}}
-\frac1{2}\sum_{k,l=1}^r[jkl]_i\frac{y_{ki}y_{li}}{y_{ji}}\bigg)
=\lambda T_jy_j \qquad  \text{ for $j\in I$}.
\end{align*}
Recall also that the limit $y_k=\lim_{i\to\infty} y_{ki}$ is positive if $k\in J$ and zero if $k\in J^c$. Furthermore, $\lim_{i\to\infty} [jkl]_i = [jkl] =0$ if $k,l\in J$ and $j\in J^c$. Thus
\begin{align}\label{lim_Ric1}
\lim_{i\to\infty}
\sum_{k,l=1}^r [jkl]_i\frac{y_{ki}}{y_{li}}
<\infty \qquad  \text{ for $j\in J$,}
\end{align}
and 
\begin{align}\label{lim_Ric2}
\lim_{i\to\infty}
\sum_{k,l=1}^r [jkl]_i\left(\frac{y_{li}y_{ji}}{y_{ki}}+\frac{y_{ki}y_{ji}}{y_{li}}
-\frac{y_{ki}y_{li}}{y_{ji}}\right)
=0 \qquad  \text{for $j\in J^c$}.
\end{align}

Given $p\in J^c$, denote by $J_p$ the set containing $p$ and all the indices $q\in J^c$ with the following property:  
there exist finite sequences $j_1,\ldots,j_\mu\in J$ and $p_1,\ldots,p_{\mu-1}\in J^c$ such that
\begin{align*}
[j_1pp_1][j_2p_1p_2][j_3p_2p_3]\cdots[j_{\mu-1}p_{\mu-2}p_{\mu-1}][j_\mu p_{\mu-1}q]>0.
\end{align*}
The space $\fm_{J_p}$ is $\Ad_K$-invariant, while for any nonempty proper subset $J_p'\subset J$, the space $\fm_{J_p'}$ is not.

In order to show that the metric $g_F$  satisfies~\eqref{PS_nec}, we need the following lemma. Passing to a subsequence if necessary, we can assume that for all~$j,k,l\in I$, the sequences 
$$
[jkl]_i\frac{y_{ji}y_{ki}}{y_{li}} 
\qquad\mbox{and}\qquad \frac{y_{ki}}{y_{li}}
$$
are monotone in $i$.

\begin{lem}\label{lem_PS_same_coord}
The following statements hold for all $a,b\in I$ and $p,q\in J^c$:
\begin{enumerate}
\item[\emph{(1)}]
If there exists $j\in J$ such that $[jpq]>0$, then
\begin{align*}
\lim_{i\to\infty}\frac{y_{pi}}{y_{qi}}=1.
\end{align*}

\item[\emph{(2)}]
If $[jaq]=0$ and $[jbq]=0$ for all $j\in J$, then
\begin{align}\label{key_lem_eq2}
\lim_{i\to\infty}[abq]_i\frac{y_{ai}y_{bi}}{y_{qi}}=0.
\end{align}
\end{enumerate}
\end{lem}

\begin{proof}
Choose $\alpha_i\in J_c$ so that $y_{\alpha_ii}=\min\{y_{qi}\,|\,q\in J^c\}$. Passing to a subsequence if necessary, we may assume that $\alpha_i$ does not depend on~$i$. This allows us to omit the index in our notation. Thus, we have
\begin{align}\label{rat_control}
\frac{y_{\alpha_ii}}{y_{q i}}=\frac{y_{\alpha i}}{y_{q i}}\le1,\qquad q\in J^c,~i\in\mathbb N.
\end{align}
Formula~\eqref{lim_Ric2} yields
\begin{align*}
\sum_{q\in J_\alpha}\sum_{k,l\in I} [qkl]_i\left(\frac{y_{li}y_{qi}}{y_{ki}}+\frac{y_{ki}y_{qi}}{y_{li}}
-\frac{y_{ki}y_{li}}{y_{qi}}\right)
\to0,\qquad i\to\infty.
\end{align*}
Regrouping the summands, we obtain
\begin{align*}
\sum_{q\in J_\alpha}\sum_{k,l\in I} [qkl]_i&\Big(\frac{y_{li}y_{qi}}{y_{ki}}+\frac{y_{ki}y_{qi}}{y_{li}}
-\frac{y_{ki}y_{li}}{y_{qi}}\Big)
\\
&=2\sum_{q\in J_\alpha}\sum_{k,l\in J}[qkl]_i\frac{y_{li}y_{qi}}{y_{ki}}-\sum_{q\in J_\alpha}\sum_{k,l\in J}[akl]_i\frac{y_{ki}y_{li}}{y_{qi}}
\\
&\hphantom{=}~+2\sum_{q\in J_\alpha}\sum_{k,l\in J^c}[qkl]_i\frac{y_{li}y_{qi}}{y_{ki}}-\sum_{q\in J_\alpha}\sum_{k,l\in J^c}[akl]_i\frac{y_{ki}y_{li}}{y_{qi}}
\\
&\hphantom{=}~+2\sum_{q\in J_\alpha}\sum_{k\in J}\sum_{l\in J^c}[qkl]_i\frac{y_{li}y_{qi}}{y_{ki}}+2\sum_{q,l\in J_\alpha}\sum_{k\in J}[qkl]_i\Big(\frac{y_{ki}y_{qi}}{y_{li}}-\frac{y_{ki}y_{li}}{y_{qi}}\Big)
\\
&\hphantom{=}~+2\sum_{q\in J_\alpha}\sum_{k\in J}\sum_{l\in J^c\setminus J_\alpha}[qkl]_i\Big(\frac{y_{ki}y_{qi}}{y_{li}}-\frac{y_{ki}y_{li}}{y_{qi}}\Big).
\end{align*}
The sixth term on the right-hand side vanishes because it is anti-symmetric in $q$ and~$l$. As a consequence,
\begin{align}\label{eq_big_aux1}
&2\sum_{q\in J_\alpha}\sum_{k\in J}\sum_{l\in I}[qkl]_i\frac{y_{li}y_{qi}}{y_{ki}}-\sum_{q\in J_\alpha}\sum_{k,l\in J}[qkl]_i\frac{y_{ki}y_{li}}{y_{qi}}\notag
\\
&\hphantom{=}~+2\sum_{q\in J_\alpha}\sum_{k,l\in J^c}[qkl]_i\frac{y_{li}y_{qi}}{y_{ki}}-\sum_{q\in J_\alpha}\sum_{k,l\in J^c}[qkl]_i\frac{y_{ki}y_{li}}{y_{qi}}\notag
\\
&\hphantom{=}~+2\sum_{q\in J_\alpha}\sum_{k\in J}\sum_{l\in J^c\setminus J_\alpha}[qkl]_i\frac{y_{ki}y_{qi}}{y_{li}}-2\sum_{q\in J_\alpha}\sum_{k\in J}\sum_{l\in J^c\setminus J_\alpha}[qkl]_i\frac{y_{ki}y_{li}}{y_{qi}}\to0,\qquad i\to\infty.
\end{align}
Let us show that the three positive terms on the left-hand side go to~0. This will partly imply assertion~(2) of the lemma.

Recall that $y_k>0$ if $k\in J$ and $y_k=0$ if $k\in J^c$. Therefore,
\begin{align*}
\lim_{i\to\infty}\sum_{q\in J_\alpha}\sum_{k\in J}\sum_{l\in I}[qkl]_i\frac{y_{li}y_{qi}}{y_{ki}}&=\sum_{q\in J_\alpha}\sum_{k\in J}\sum_{l\in I}[qkl]\frac{y_{l}y_{q}}{y_{k}}=0.
\end{align*}
Formula~\eqref{lim_Ric1} implies that
\begin{align}\label{ratio_same_module}
\frac{y_{ai}}{y_{bi}}<c\in\mathbb R
\end{align}
whenever $a,b\in J_{\omega}$ for some $\omega\in J^c$. Together with~\eqref{rat_control}, this yields
$$
\frac{y_{qi}}{y_{ki}}=\frac{y_{qi}}{y_{\alpha i}}\frac{y_{\alpha i}}{y_{ki}}<c
$$
for $q\in J_\alpha$ and $k\in J^c$. Consequently,
\begin{align*}
\lim_{i\to\infty}\sum_{q\in J_\alpha}\sum_{k,l\in J^c}[qkl]_i\frac{y_{li}y_{qi}}{y_{ki}}&\le c\sum_{q\in J_\alpha}\sum_{k,l\in J^c}[qkl]y_l=0,
\\
\lim_{i\to\infty}\sum_{q\in J_\alpha}\sum_{k\in J}\sum_{l\in J^c\setminus J_\alpha}[qkl]_i\frac{y_{ki}y_{qi}}{y_{li}}
&\le
c\sum_{q\in J_\alpha}\sum_{k\in J}\sum_{l\in J^c\setminus J_\alpha}[qkl]y_k.
\end{align*}
In the second line, on the right-hand side, since $q$ lies in $J_\alpha$ and $l$ lies outside of $J_\alpha$, the structure constant $[qkl]$ equals~0. We conclude that all three positive terms in~\eqref{eq_big_aux1} vanish. This means that
\begin{align*}
\sum_{q\in J_\alpha}\sum_{k,l\in J}[qkl]_i\frac{y_{ki}y_{li}}{y_{qi}}\notag
+\sum_{q\in J_\alpha}\sum_{k,l\in J^c}[qkl]_i\frac{y_{ki}y_{li}}{y_{qi}}+2\sum_{q\in J_\alpha}\sum_{k\in J}\sum_{l\in J^c\setminus J_\alpha}[qkl]_i\frac{y_{ki}y_{li}}{y_{qi}}\to0,\qquad i\to\infty.
\end{align*}
As a result,~\eqref{key_lem_eq2} holds
whenever $q\in J_\alpha$ and the pair $(a,b)$ lies outside the set $(J\times J_\alpha)\cup(J_\alpha\times J)$. If $q\in J_\alpha$ and $(a,b)\in J\times J_\alpha$, then
\begin{align}\label{aux1_main_lem}
\lim_{i\to\infty}[abq]_i\frac{y_{ai}y_{bi}}{y_{qi}}\le c[abq]y_a.
\end{align}
Similarly, if $q\in J_\alpha$ and $(a,b)\in J_\alpha\times J$, then
\begin{align}\label{aux2_main_lem}
\lim_{i\to\infty}[abq]_i\frac{y_{ai}y_{bi}}{y_{qi}}\le c[abq]y_b.
\end{align}
Now, we can easily conclude that assertion~(2) of the lemma holds for $q\in J_\alpha$.

Choose $\beta_i$ so that
$y_{\beta_ii}=\min\{y_{qi}\,|\,q\in J^c\setminus J_\alpha\}$. Again, passing to a subsequence if necessary, we may assume that $\beta_i$ does not depend on~$i$ and omit the index in the notation. Thus,
\begin{align}\label{rat_con2}
\frac{y_{\beta i}}{y_{qi}}\le1,\qquad q\in J^c\setminus J_\alpha,~i\in\mathbb N.
\end{align}
Let us sum~\eqref{lim_Ric2} over all $j=q\in J_\beta$ and regroup the terms as above. Noting that $y_k=\lim_{i\to\infty}y_{ki}$ is positive when $k\in J$ and zero when $k\in J^c$, we find
\begin{align}\label{eq_big_aux2}
-\sum_{q\in J_\beta}\sum_{k,l\in J}&[qkl]_i\frac{y_{ki}y_{li}}{y_{qi}}+2\sum_{q\in J_\beta}\sum_{k,l\in J^c}[qkl]_i\frac{y_{li}y_{qi}}{y_{ki}}-\sum_{q\in J_\beta}\sum_{k,l\in J^c}[qkl]_i\frac{y_{ki}y_{li}}{y_{qi}}\notag
\\
&+2\sum_{q\in J_\beta}\sum_{k\in J}\sum_{l\in J^c\setminus J_\beta}[qkl]_i\frac{y_{ki}y_{qi}}{y_{li}}-2\sum_{q\in J_\beta}\sum_{k\in J}\sum_{l\in J^c\setminus J_\beta}[qkl]_i\frac{y_{ki}y_{li}}{y_{qi}}\to0,\qquad i\to\infty.
\end{align}
Our next step is to show that the two positive terms on the left-hand side go to~0.

Observe that
\begin{align*}
\sum_{q\in J_\beta}\sum_{k,l\in J^c}[qkl]_i\frac{y_{li}y_{qi}}{y_{ki}}
&=\sum_{q\in J_\beta}\sum_{k\in J^c\setminus J_\alpha}\sum_{l\in J^c}[qkl]_i\frac{y_{li}y_{qi}}{y_{ki}}
\\
&\hphantom{=}~+\sum_{k\in J_\beta}\sum_{q,l\in J_\alpha}[qkl]_i\frac{y_{li}y_{ki}}{y_{qi}}
+\sum_{k\in J_\beta}\sum_{q\in J_\alpha}\sum_{l\in J^c\setminus J_\alpha}[qkl]_i\frac{y_{li}y_{ki}}{y_{qi}}.
\end{align*}
Formulas~\eqref{ratio_same_module} and~\eqref{rat_con2} yield
\begin{align*}
\lim_{i\to\infty}\sum_{q\in J_\beta}\sum_{k\in J^c\setminus J_\alpha}\sum_{l\in J^c}[qkl]_i\frac{y_{li}y_{qi}}{y_{ki}}&\le c\sum_{q\in J_\beta}\sum_{k\in J^c\setminus J_\alpha}\sum_{l\in J^c}[qkl]y_{l}=0,
\\
\lim_{i\to\infty}\sum_{k\in J_\beta}\sum_{q,l\in J_\alpha}[qkl]_i\frac{y_{li}y_{ki}}{y_{qi}}&\le c\sum_{k\in J_\beta}\sum_{q,l\in J_\alpha}[qkl]y_k=0.
\end{align*}
At the same time, noting that $[jaq]=0$ if $j\in J$, $a\in J^c\setminus J_\alpha$ and $q\in J_\alpha$, we see that
\begin{align*}
\lim_{i\to\infty}\sum_{q\in J_\alpha}\sum_{k\in J_\beta}\sum_{l\in J^c\setminus J_\alpha}[qkl]_i\frac{y_{li}y_{ki}}{y_{qi}}=0
\end{align*}
because assertion~(2) of the lemma holds for $q\in J_\alpha$. Thus, the first positive term in~\eqref{eq_big_aux2} goes to~0. Similarly,
\begin{align*}
\sum_{q\in J_\beta}&\sum_{k\in J}\sum_{l\in J^c\setminus J_\beta}[qkl]_i\frac{y_{ki}y_{qi}}{y_{li}}
\\
&=\sum_{q\in J_\beta}\sum_{k\in J}\sum_{l\in J^c\setminus(J_\alpha\cup J_\beta)}[qkl]_i\frac{y_{ki}y_{qi}}{y_{li}}+\sum_{q\in J_\alpha}\sum_{k\in J}\sum_{l\in J_\beta}[qkl]_i\frac{y_{ki}y_{li}}{y_{qi}}\to0,\qquad i\to\infty.
\end{align*}
Now, formula~\eqref{eq_big_aux2} yields
\begin{align*}
\sum_{q\in J_\beta}\sum_{k,l\in J}&[qkl]_i\frac{y_{ki}y_{li}}{y_{qi}}+\sum_{q\in J_\beta}\sum_{k,l\in J^c}[qkl]_i\frac{y_{ki}y_{li}}{y_{qi}}+2\sum_{q\in J_\beta}\sum_{k\in J}\sum_{l\in J^c\setminus J_\beta}[qkl]_i\frac{y_{ki}y_{li}}{y_{qi}}\to0,\qquad i\to\infty.
\end{align*}
This implies that~\eqref{key_lem_eq2} holds for $q\in J_\beta$ and $(a,b)$ outside of $(J\times J_\beta)\cup(J_\beta\times J)$. If $q\in J_\beta$ and $(a,b)\in J\times J_\beta$ or $(a,b)\in J_\beta\times J$, then one can write estimates analogous to~\eqref{aux1_main_lem} or~\eqref{aux2_main_lem}. We conclude that assertion~(2) of the lemma holds for $q\in J_\beta$.

Choose $\gamma_i$ so that
$y_{\gamma_ii}=\min\{y_{qi}\,|\,q\in J^c\setminus (J_\alpha\cup J_\beta)\}$. Pass to a subsequence if necessary to ensure that $\gamma_i$ does not depend on $i$ and omit $i$ from the notation. We can argue as above to show that assertion~(2) of the lemma holds when $q\in J_\gamma$. Continuing in this way, we eventually obtain it for all $q\in J^c$. Now, let us prove assertion~(1).

Given $p\in J^c$, denote
$$
J_p^*=\{p\}\cup\{q\in J^c\,|\,[jpq]>0~\mbox{for some}~j\in J\}.
$$
Inequality~\eqref{ratio_same_module} implies that the sequence $\frac{y_{pi}}{y_{qi}}$ converges to some~$\gamma_{pq}\in(0,\infty)$ for each $q\in J_p^*$, i.e.,
\begin{align*}
\lim_{i\to\infty}\frac{y_{pi}}{y_{qi}}=\gamma_{pq}\in(0,\infty),\qquad q\in J_p^*.
\end{align*}
Our goal is to show that all these $\gamma_{pq}$ equal~1. Choose $p_1\in J^c$ and $p_0\in J_{p_1}^*$ such that $\gamma_{p_1p_0}$ is the smallest possible, i.e.,
\begin{align*}
\gamma_{p_1p_0}=\min\{\gamma_{pq}\,|\,p\in J^c~\mbox{and}~q\in J_p^*\}.
\end{align*}
Since $\gamma_{pq}=\frac1{\gamma_{qp}}$ and $q\in J_p^*$ if and only if $p\in J_q^*$, it suffices to prove that $\gamma_{p_1p_0}\ge1$. Assume that $\gamma_{p_1p_0}<1$. Obviously, $p_1\ne p_0$. Using~\eqref{lim_Ric2} with $j=p_1$, we obtain
\begin{align*}
2&\sum_{k,l\in J}[p_1kl]_i\frac{y_{p_1i}y_{li}}{y_{ki}}-\sum_{k,l\in J}[p_1kl]_i\frac{y_{ki}y_{li}}{y_{p_1i}}
\\
&+2\sum_{k,l\in J^c}[p_1kl]_i\frac{y_{p_1i}y_{ki}}{y_{li}}-\sum_{k,l\in J^c}[p_1kl]_i\frac{y_{ki}y_{li}}{y_{p_1i}}
\\
&+2\sum_{k\in J}\sum_{l\in J^c}[p_1kl]_i\frac{y_{p_1i}y_{li}}{y_{ki}}+2\sum_{k\in J}\sum_{l\in J_{p_1}^*}y_{ki}[p_1lk]_i\Big(\frac{y_{p_1i}}{y_{li}}-\frac{y_{li}}{y_{p_1i}}\Big)
\\
&+2\sum_{k\in J}\sum_{l\in J^c\setminus J_{p_1}^*}[p_1kl]_i\frac{y_{ki}y_{p_1i}}{y_{li}}
-2\sum_{k\in J}\sum_{l\in J^c\setminus J_{p_1}^*}[p_1kl]_i\frac{y_{ki}y_{li}}{y_{p_1i}}\to 0,\qquad i\to\infty.
\end{align*}
In light of~\eqref{ratio_same_module} and assertion~(2), which we have already proven, this implies
\begin{align*}
\sum_{k\in J}\sum_{l\in J_{p_1}^*}y_k[p_1kl](\gamma_{p_1l}-\gamma_{lp_1})=0.
\end{align*}
Because $p_0\in J_{p_1}^*$, there exists $k_1\in J$ such that~$[k_1p_1p_0]>0$. Therefore, the above sum contains the term
\begin{align*}
y_{k_1}[k_1p_1p_0](\gamma_{p_1p_0}-\gamma_{p_0p_1})=y_{k_1}[k_1p_1p_0]\big(\gamma_{p_1p_0}-\tfrac1{\gamma_{p_1p_0}}\big)<0.
\end{align*}
Since this sum equals zero, there exist $k_1'\in J$ and $p_2\in J_{p_1}^*$ such that $[k_1'p_1p_2]>0$ and
\begin{align*}
y_{k_1'}[k_1'p_1p_2](\gamma_{p_1p_2}-\gamma_{p_2p_1})=y_{k_1'}[k_1'p_1p_2]\big(\tfrac1{\gamma_{p_2p_1}}-\gamma_{p_2p_1}\big)>0.
\end{align*}
This formula implies that $\gamma_{p_2p_1}<1$. Clearly, $p_2\ne p_1$. Also, $p_2\ne p_0$ because
$$
\lim_{i\to\infty}\frac{y_{p_2i}}{y_{p_0i}}=\gamma_{p_2p_1}\gamma_{p_1p_0}<1.
$$

Using~\eqref{lim_Ric2} as above but with $j=p_2$, we obtain
\begin{align*}
\sum_{k\in J}\sum_{l\in J_{p_2}^*}y_k[p_2kl](\gamma_{p_2l}-\gamma_{lp_2})=0.
\end{align*}
It is possible to find $k_2\in J$ such that~$[k_2p_2p_1]>0$. Therefore, the above sum contains the term
\begin{align*}
y_{k_2}[k_2p_2p_1](\gamma_{p_2p_1}-\gamma_{p_1p_2})<0.
\end{align*}
Since this sum equals zero, there exist $k_2'\in J$ and $p_3\in J_{p_2}^*$ such that $[k_2'p_2p_3]>0$ and
\begin{align*}
y_{k_2'}[k_2'p_2p_3](\gamma_{p_2p_3}-\gamma_{p_3p_2})>0.
\end{align*}
We conclude that $\gamma_{p_3p_2}<1$. Evidently, $p_3\ne p_2$. We also have $p_3\ne p_1$ and $p_3\ne p_0$ because
$$
\lim_{i\to\infty}\frac{y_{p_3i}}{y_{p_1i}}=\gamma_{p_3p_2}\gamma_{p_2p_1}<1\qquad \mbox{and} \qquad \lim_{i\to\infty}\frac{y_{p_3i}}{y_{p_0i}}=\gamma_{p_3p_2}\gamma_{p_2p_1}\gamma_{p_1p_0}<1.
$$

Continuing in this way, we obtain an infinite sequence $(p_n)$ of distinct numbers in $J^c$. This is obviously impossible. Thus, $\gamma_{pq}$ cannot be less than~1.
\end{proof}

We now show that the metric $g_F$ given by~\eqref{gF_def} satisfies~\eqref{PS_nec}. The Killing form of $K$ restricted to $\fm_j$ equals $-\bar b_jQ_{|\fm_j}$ for all $j\in J$ with
\begin{align*}
\bar b_j=b_j-\frac1{d_j}\sum_{k,l\in J^c}[jkl].
\end{align*}
Using~\eqref{PS_sequence}, \eqref{RicciDiagonal} and Lemma~\ref{lem_PS_same_coord}, we see that for all $j\in J$
\begin{align*}
\lambda T_{|\fm_j}&=
\lim_{i\to\infty}\bigg(\frac{b_{ji}}2
-\frac1{2d_j} \sum_{k,l=1}^r [jkl]_i\frac{y_{li}}{y_{ki}}
+\frac1{4d_j}\sum_{k,l=1}^r[jkl]_i\frac{y_{ki}y_{li}}{y_{ji}^2}\bigg)Q_{|\fm_j}
\\
&=
\bigg(\frac{\bar b_{j}}2
-\frac1{2d_j} \sum_{k,l\in J}[jkl]\frac{y_{l}}{y_{k}}
+\frac1{4d_j}\sum_{k,l\in J}[jkl]\frac{y_{k}y_{l}}{y_{j}^2}\bigg)Q_{|\fm_j}=\Ric(g_F)_{|\fm_j}.
\end{align*}
It remains to show that the off-diagonal components of $\Ric(g_F)$ coincide with those of  $\lambda T$. 
For this, let $(e_\mu^{li})$ be a $Q$-orthonormal basis of $\fm_l^i$, and let the subscript $\fm_m^i$ denote projection onto~$\fm_m^i$.
We may assume that $(e_\mu^{li})$ converges, as $i\to\infty$, to a basis $(e_\mu^l)$ of $\fm_l$ for every~$l\in I$.

Formula \eqref{RiccOffDiag} implies that
\begin{align*}
\Ric(g_i)(e_p^{ji},e_q^{ki})&=\sum_{l,m=1}^r\Big(\frac12-\frac{y_{mi}}{2y_{li}}+\frac{y_{li}y_{mi}}{4y_{ji}y_{ki}}\Big)
\sum_{\mu=1}^{d_l}Q\big([e_p^{ji},e_\mu^{li}]_{\fm_m^i},[e_q^{ki},e_\mu^{li}]_{\fm_m^i}\big)
\end{align*}
for all $j,k\in I$ with $j\ne k$, $p=1,\ldots,d_j$ and  $q=1,\ldots,d_k$, and hence 
\begin{align}\label{lim_RicOD}
\lim_{i\to\infty}\sqrt{y_{ji}y_{ki}}\sum_{l,m=1}^r\Big(\frac12-\frac{y_{mi}}{2y_{li}}+\frac{y_{li}y_{mi}}{4y_{ji}y_{ki}}\Big)
\sum_{\mu=1}^{d_l}Q\big([e_p^{ji},e_\mu^{li}]_{\fm_m^i},[e_q^{ki},e_\mu^{li}]_{\fm_m^i}\big)=\lambda T(e_p^{j},e_q^{k})\sqrt{y_jy_k}.
\end{align}
The Cauchy--Schwarz inequality yields
\begin{align*}
\big|Q\big([e_p^{ji},e_\mu^{li}]_{\fm_m^i},[e_q^{ki},e_\mu^{li}]_{\fm_m^i}\big)\big|
&\le\big|[e_p^{ji},e_\mu^{li}]_{\fm_m^i}\big|_Q\big|[e_q^{ki},e_\mu^{li}]_{\fm_m^i}\big|_Q\notag \\
&\le\sqrt{[jlm]_i}\sqrt{[klm]_i}\le\max\{[jlm]_i,[klm]_i\}.
\end{align*}
If $j,k,l\in J$ and $m\in J^c$, then
\begin{align*}
\lim_{i\to\infty}\big|Q\big([e_p^{ji},e_\mu^{li}]_{\fm_m^i},[e_q^{ki},e_\mu^{li}]_{\fm_m^i}\big)\big|
&\le\lim_{i\to\infty}\max\{[jlm]_i,[klm]_i\}=\max\{[jlm],[klm]\}=0.
\end{align*}
By Lemma~\ref{lem_PS_same_coord},
\begin{align*}
\lim_{i\to\infty}\frac{\big|Q\big([e_p^{ji},e_\mu^{mi}]_{\fm_l^i},[e_q^{ki},e_\mu^{mi}]_{\fm_l^i}\big)\big|}{y_{mi}}
&\le\lim_{i\to\infty}\frac{\max\{[jlm]_i,[klm]_i\}}{y_{mi}}=0.
\end{align*}
Also, by the same lemma, if $j,k\in J$, $l\in J^c$ and $m\in J_l$, then
\begin{align*}
\lim_{i\to\infty}\Big(\frac12-\frac{y_{mi}}{2y_{li}}\Big)=0.
\end{align*}
At the same time, if $j,k\in J$, $l\in J^c$ and $m\notin J_l$, then
\begin{align*}
\lim_{i\to\infty}\frac{\big|Q\big([e_p^{ji},e_\mu^{mi}]_{\fm_l^i},[e_q^{ki},e_\mu^{mi}]_{\fm_l^i}\big)\big|y_{li}}{y_{mi}}
&\le\lim_{i\to\infty}\frac{\max\{[jlm]_i,[klm]_i\}y_{li}}{y_{mi}}=0.
\end{align*}
Using~\eqref{lim_RicOD} along with these formulas, we find, for $j,k\in J$ with $j\ne k$, that
\begin{align*}
\lambda T(e_p^{j},e_q^{k})&=\lim_{i\to\infty}\sum_{l,m=1}^r\Big(\frac12-\frac{y_{mi}}{2y_{li}}+\frac{y_{li}y_{mi}}{4y_{ji}y_{ki}}\Big)
\sum_{r=1}^{d_l}Q\big([e_p^{ji},e_r^{li}]_{\fm_m^i},[e_q^{ki},e_r^{li}]_{\fm_m^i}\big)
\\
&=
\sum_{l,m\in J}\Big(\frac12-\frac{y_{m}}{2y_{l}}+\frac{y_{l}y_{m}}{4y_{j}y_{k}}\Big)
\sum_{r=1}^{d_l}Q\big([e_p^{j},e_r^{l}]_{\fm_m},[e_q^{k},e_r^{l}]_{\fm_m}\big)
\\
&=\Ric(g_F)(e_p^{j},e_q^{k}).
\end{align*}
Thus $\Ric(g_F)=\lambda T_{|K/H}$.
\end{proof}

Let us briefly explain how we prove the corollary in the introduction. If $(g_i)$ is a 0--Palais--Smale sequence, then \tref{thm_PSnec} implies that there exist an intermediate subgroup $K$ and a metric $g\in\M(K/H)$ with $\Ric(g)=0$. According to \cite[Theorem~7.61]{AB87}, $g$ must be flat. As discussed in~\cite[Section~1]{PZ22}, we can extend $g$ to a metric on $\bar K/H$, which is also flat (here $\bar K$ is the closure of $K$). Since $\bar K$ and $\bar K/H$ are compact, \cite[Theorem~1.84]{AB87} implies that $\bar K/H$ is a torus. 

\begin{rem} (a) The limit value $y_j$ is zero if $j\in J^c$. Thus, one could say that the metrics $g_i$ blow up the base $G/K$ in the limit. However, if we normalize these metrics appropriately, their limit on $G/K$ becomes a meaningful geometric object. More precisely, define
\begin{align*}
\tilde g_i=((\tilde y_{1i},\ldots,\tilde y_{ri}),D_i)=\Big(\Big(\frac{y_{1i}}{\min_{j\in J^c}y_{ji}},\ldots,\frac{y_{ri}}{\min_{j\in J^c}y_{ji}}\Big),D_i\Big).
\end{align*}
Passing to a subsequence if necessary, we may assume that $\tilde y_{ji}$ converges to some $\tilde y_j\in[1,\infty]$ as $i\to\infty$ whenever $j\in J^c$. Thus, the restriction of $\tilde g_i$ to $\fk^\perp$ converges to
\begin{align}\label{def_tilde_g}
\tilde g=\sum_{j\in J^c}\frac1{\tilde y_j}Q_{|\fm_j},
\end{align}
where we interpret $\frac1{\tilde y_j}$ as 0 if $\tilde y_j=\infty$. Lemma~\ref{lem_PS_same_coord} implies that~(\ref{def_tilde_g}) defines a positive-semi-definite $K$-invariant tensor field on~$G/K$. However, this restriction may fail to be  positive definite.

(b) The existence of the global maximum in \tref{Ex_max} is also a consequence of the characterization of Palais--Smale sequences in \tref{main_1}. Indeed, if $\beta_\fk-\alpha_\fk>0$ at the highest level~$\alpha_\fk$, and $\alpha_\fk$ is assumed, then metrics $g$ with $S(g)>\alpha_\fk$ exist along a canonical variation, as follows from~\eqref{scal_g}. Such a  subalgebra, where $\alpha_\fk$ is assumed, exists due to  \cite[Proposition~3.7]{PZ22}. Therefore, the set $\{g\in\M_T\mid S(g)>\alpha_\fk+\e\}$ is nonempty for some $\e>0$, and Theorem~\ref{main_1} implies that this set is compact.
\end{rem}

In order to use mountain pass techniques, we need to show that $S^{-1}((c,\infty))\cap\M_T$ has two connected components for certain values $c\in\mathbb R$. We expect this to happen if $c$ is slightly smaller than $\alpha_\fk$ when $\beta_\fk-\alpha_\fk<0$. The following proposition proves this under additional assumptions.

\begin{prop}\label{neg_der}
Consider a compact homogeneous space $G/H$ such that the modules $\fm_i$ are pairwise inequivalent. Let $K$ be an intermediate subgroup with Lie algebra~$\fk$ such that $G/K$ is isotropy irreducible. Assume that $\alpha_{\fk'}<\alpha_\fk$ for every intermediate subalgebra $\fk'$ contained in $\fk$ as a proper subset. Furthermore, assume that $\beta_\fk-\alpha_\fk<0$. Given $\delta>0$, there exist $\delta_-,\delta_+,\epsilon>0$ such that $\delta_-<\delta_+\le\delta$ and $S(y)\le\alpha_{\fk}-\epsilon$ for all $y\in T_{\delta_+}(\Delta_\fk)\setminus T_{\delta_-}(\Delta_\fk)$.
\end{prop}

\begin{proof}
Take $y=(y_1,\ldots,y_r)\in\Delta$ and let $\Delta_\fk=(\Delta_J,\fk)$. We have
\begin{align}\label{est_fla1}
	S(y)&=S(y_{|\fk\cap\fm})+\frac12\sum_{i\in J^c}d_ib_iy_i-\frac14\sum_{i,j,k\in J^c} [ijk]\frac{y_iy_j}{y_k} \notag
	\\
	&\hphantom{=}~-\frac14\sum_{i\in J}\sum_{j,k\in J^c}[ijk]\Big(\frac{y_ky_j}{y_i}+y_i\Big(\frac{y_j}{y_k}+\frac{y_k}{y_j}-2\Big)\Big) \notag
	\\
	&\le\alpha_\fk \tr_{y_{|\fk\cap\fm}}T_{|\fk\cap\fm}+\frac12\sum_{i\in J^c}d_ib_iy_i
	-\frac14\sum_{i,j,k\in J^c} [ijk]\frac{y_iy_j}{y_k} \notag
	\\
	&\hphantom{=}~
	-\frac14\sum_{i\in J}\sum_{j,k\in J^c}y_i[ijk]\Big(\frac{y_j}{y_k}+\frac{y_k}{y_j}-2\Big).
	\end{align}
Choose $p,q\in J^c$ so that
	\begin{align*}
	y_p=\min_{i\in J^c}y_i\qquad\mbox{and}\qquad y_q=\max_{i\in J^c}y_i.
	\end{align*}
If $y_p=y_q$, then $y_i$ are the same for all $i\in J^c$. In this case, the restriction $y_{|\fk^\perp}$ defines a $K$-invariant metric on~$G/K$ with scalar curvature
\begin{align*}
S(y_{|\fk^\perp})=\frac12\sum_{i\in J^c}d_ib_iy_i
-\frac14\sum_{i,j,k\in J^c} [ijk]\frac{y_iy_j}{y_k}=\sum_{i\in J^c}y_p\Big(\frac{d_ib_i}2
-\frac14\sum_{j,k\in J^c}[ijk]\Big).
\end{align*}
Moreover,
$$
S\big(y_{|\fk^\perp}\big)=\beta_\fk\tr_{y_{|\fk^\perp}}T_{|\fk^\perp}=\beta_\fk\sum_{i\in J^c}d_iT_iy_p.
$$

The equality $y_p=y_q$ does not hold in general. One may view the quantity $\frac{y_q}{y_p}-1$ as a measure of the deviation of $y_{|\fk^\perp}$ from being a $K$-invariant metric on~$G/K$. First, we consider the situation where this deviation is small. More precisely, suppose that
	\begin{align*}
	\frac{y_q}{y_p}-1\le\Theta,\qquad \mbox{where}\qquad \Theta=\frac{\alpha_\fk-\beta_\fk}{\sum_{i\in J^c}d_ib_i+1}\min_{i\in J^c}d_iT_i>0.
	\end{align*}
Since $G/K$ is isotropy irreducible,
\begin{align*}
	S(y)&\le \alpha_\fk \tr_{y_{|\fk\cap\fm}}T_{|\fk\cap\fm}+\frac12\sum_{i\in J^c}d_ib_iy_i
	-\frac18\sum_{i,j,k\in J^c}y_i[ijk]\Big(\frac{y_j}{y_k}+\frac{y_k}{y_j}\Big)
	\\
	&\le \alpha_\fk \tr_{y_{|\fk\cap\fm}}T_{|\fk\cap\fm}+\sum_{i\in J^c}y_p\Big(\frac{d_ib_i}2
	-\frac14\sum_{j,k\in J^c}[ijk]\Big)+\sum_{i\in J^c}(y_i-y_p)\Big(\frac{d_ib_i}2
	-\frac14\sum_{j,k\in J^c}[ijk]\Big)
	\\
	&=\alpha_\fk \tr_{y_{|\fk\cap\fm}}T_{|\fk\cap\fm}+\beta_\fk \sum_{i\in J^c}d_iT_iy_p+\sum_{i\in J^c}y_p\Big(\frac{y_i}{y_p}-1\Big)\Big(\frac{d_ib_i}2
	-\frac14\sum_{j,k\in J^c}[ijk]\Big)
	\\
	&\le \alpha_\fk \tr_{y_{|\fk\cap\fm}}T_{|\fk\cap\fm}+\beta_\fk\tr_{y_{|\fk^\perp}}T_{|\fk^\perp}+\sum_{i\in J^c}\frac{d_ib_iy_p}2\Big(\frac{y_i}{y_p}-1\Big).
	\end{align*}
If $y\in T_{\delta_+}(\Delta_\fk)\setminus T_{\delta_-}(\Delta_\fk)$ for some $\delta_-,\delta_+>0$ such that $\delta_-<\delta_+$, then $\delta_-<y_q\le\delta_+$ and
	\begin{align}\label{est_fla2}
	S(y)&\le \alpha_\fk - (\alpha_\fk-\beta_\fk)\tr_{y_{|\fk^\perp}}T_{|\fk^\perp}+\sum_{i\in J^c}\frac{d_ib_iy_p}2\Big(\frac{y_q}{y_p}-1\Big) \notag
	\\
	&\le \alpha_\fk - (\alpha_\fk-\beta_\fk)d_qT_qy_q+\frac12\sum_{i\in J^c}\Theta d_ib_iy_p \notag
	\\
	&\le \alpha_\fk - (\alpha_\fk-\beta_\fk)d_qT_q\delta_-+\frac{\alpha_\fk-\beta_\fk}2\min_{i\in J^c}d_iT_i\,\delta_+ \notag
	\\
	&\le \alpha_\fk - (\alpha_\fk-\beta_\fk)\min_{i\in J^c}d_iT_i\Big(\delta_--\frac{\delta_+}2\Big).
	\end{align}
Setting $\delta_-=\frac{3\delta+}4$, we conclude that	
	\begin{align}\label{aux_negder}
	S(y)&\le \alpha_\fk - \frac{\delta_+}4(\alpha_\fk-\beta_\fk)\min_{i\in J^c}d_iT_i.
	\end{align}
	
Now, let us consider the situation where the deviation of $y_{|\fk^\perp}$ from being a $K$-invariant metric on $G/K$ is large. More precisely, suppose that
	\begin{align*}
	\frac{y_q}{y_p}-1>\Theta.
	\end{align*}
Because $G/K$ is isotropy irreducible, there exist 
finite sequences $j_1,\ldots,j_\mu\in J$ and $p_1,\ldots,p_{\mu-1}\in J^c$ such that
\begin{align*}
[j_1pp_1][j_2p_1p_2][j_3p_2p_3]\cdots[j_{\mu-1}p_{\mu-2}p_{\mu-1}][j_\mu p_{\mu-1}q]>0.
\end{align*}
We may assume that $p,p_1,\ldots,p_{\mu-1}$ are pairwise distinct.
The formula
\begin{align*}
\frac{y_q}{y_{p_{\mu-1}}}\cdot\frac{y_{p_{\mu-1}}}{y_{p_{\mu-2}}}\cdot\,\cdots\,\cdot\frac{y_{p_2}}{y_{p_1}}\cdot\frac{y_{p_1}}{y_p}=\frac{y_q}{y_p}>1+\Theta
\end{align*}
implies that we can find $m\in J$ and $u,v\in J^c$ satisfying
\begin{align*}
[muv]>0\qquad\mbox{and}\qquad \frac{y_v}{y_u}-1>(1+\Theta)^{1/\mu}-1\ge\Theta_*,
\end{align*}
where $\Theta_*=(1+\Theta)^{1/|J^c|}-1$.
If $y\in T_{\delta_+}(\Delta_\fk)\setminus T_{\delta_-}(\Delta_\fk)$, then
\begin{align}\label{est_fla3}
S(y)&\le\alpha_\fk \tr_{y_{|\fk\cap\fm}}T_{|\fk\cap\fm}+\frac12\sum_{i\in J^c}d_ib_iy_i
-\frac14\sum_{i,j,k\in J^c}[ijk]\frac{y_iy_j}{y_k} \notag
\\
&\hphantom{=}~-\frac14\sum_{i\in J}\sum_{j,k\in J^c}y_i[ijk]\Big(\frac{y_j}{y_k}+\frac{y_k}{y_j}-2\Big) \notag
\\
&\le\alpha_\fk+\frac12\sum_{i\in J^c}d_ib_iy_i-\frac14y_m[muv]\Big(\frac{y_u}{y_v}+\frac{y_v}{y_u}-2\Big) \notag
\\
&\le\alpha_\fk+\frac{\delta_+}2\sum_{i\in J^c}d_ib_i-y_m\frac{\theta\Theta_*^2}{4(1+\Theta_*)}
\end{align}
with $\theta=\min\{[ijk]\,|\,[ijk]>0\}$. Proposition~\ref{limit_subalgebra}, together with the hypothesis that $\alpha_{\fk'}<\alpha_\fk$ for $\fk'\subset\fk$,  yields the existence of $\delta_0,\epsilon_0>0$ such that $S(y)<\alpha_\fk-\epsilon_0$ whenever $\min_{i\in J}y_i<\delta_0$.  On the other hand, if 
\begin{align*}
\min_{i\in J}y_i\ge\delta_0\qquad\mbox{and}\qquad\delta_+\le\frac{\delta_0\theta\Theta_*^2}{4(1+\Theta_*)\big(\sum_{i\in J^c}d_ib_i+1\big)},
\end{align*}
then
\begin{align*}
S(y)\le\alpha_\fk-\delta_0\frac{\theta\Theta_*^2}{8(1+\Theta_*)}.
\end{align*}
Recalling~\eqref{aux_negder}, we conclude that the assertion of Proposition~\ref{neg_der} holds if $\delta_+$ is sufficiently small, $\delta_-=\frac{3\delta_+}4$, and
\begin{align*}
\epsilon=\min\bigg\{\frac{\delta_+}4(\alpha_\fk-\beta_\fk)\min_{i\in J^c}d_iT_i,\epsilon_0,\frac{\delta_0\theta\Theta_*^2}{8(1+\Theta_*)}\bigg\}.
\end{align*}
\end{proof}

\begin{rem}
We do not know if \pref{neg_der} holds without the assumption that $G/K$ is isotropy irreducible, or that the summands are inequivalent. As we will see in Lemma~\ref{neg_der_2}, one can remove the former assumption at least in some cases. A generalization of \pref{neg_der} would be necessary if one wanted to develop a general theory for the existence of saddle points, as was done in the Einstein case.
\end{rem}

We end this section with a few observations concerning degenerate critical points of the scalar curvature functional. Denote by $\mathcal T$  the space of $G$-invariant symmetric (0,2)-tensor fields on~$M$. We can regard the Ricci curvature as a map $\Ric\colon\M\to\mathcal T$.

\begin{prop}\label{Morse}
The following statements hold true:
\begin{enumerate}
\item[\emph{(1)}] 
If $g$ is a critical point of $S_{|\M_T}$, then
$$
\Hess(S_{|\M_T})_g(X,Y)=-\ml d\Ric_g(X),Y\mr_g \qquad \text{ for all } \qquad X,Y\in T_g(\M_T).
$$

\item[\emph{(2)}] Assume that $\Ric(g)=cT$ for some $g\in \M_T$ and $c>0$. Then the metric $g$ is a non-degenerate critical point of $S_{|\M_T}$  if and only if  $d\Ric_g$ has rank $\dim\M-1$.

\item[\emph{(3)}] The map $S\colon\M_T\to \R$ is a Morse function for all positive-definite $T$ outside a set of Lebesgue measure zero in~$\mathcal T$.
\end{enumerate}
\end{prop}

\begin{proof}
We start with the proof of statement~(1). The well-known identity $dS_g(X)=-\ml \Ric(g), X\mr_g$ implies
\begin{align*}
\Hess(S)_g(X,Y)&=\nabla_X(dS_g)(Y)
\\
&=\nabla_X(dS_g(Y))-dS_g(\nabla_XY)
=-\nabla_X\ml\Ric(g),Y\mr_g+\ml\Ric(g),\nabla_XY\mr_g
\\
&=-\ml\nabla_X\Ric(g),Y\mr_g-\ml\Ric(g),\nabla_XY\mr_g+\ml\Ric(g),\nabla_XY\mr_g=-\ml\nabla_X\Ric(g),Y\mr_g.
\end{align*}
On the right-hand side, the covariant derivative $\nabla$ is taken with respect to the metric induced by $g$ on the tensor bundle of~$\mathcal T$, and hence $\nabla_X\Ric(g)=d\Ric_g(X)$.

Next, we prove statement~(2). Observe that $d\Ric_g(g)=0$ since $\Ric(\lambda g)=\Ric(g)$ for all $\lambda>0$. Furthermore, the ray $\{\lambda g\,|\,\lambda>0\}$ is transverse to~$\M_T$ since $S(g)=\pro{g}{\Ric(g)}_g$ and $S(g)=c>0$.  In light of statement~(1), this implies statement~(2).

Finally, consider the smooth manifolds
$$
\widetilde{\M} =\M/\{g\sim \lambda g \text{ for all } \lambda>0\}\qquad\mbox{and}\qquad \widetilde{\mathcal T} =\mathcal T/\{h\sim \lambda h \text{ for all } \lambda\in\mathbb R\setminus\{0\}\}.
$$
Denote by $[g]$ and $[T]$ the corresponding equivalence classes.
The map $\Ric$ induces a map $\widetilde{\Ric}:\widetilde\M\to\widetilde{\mathcal T}$.
Clearly, $[T]$ is a regular value of $\widetilde{\Ric}$ if and only if $d\,\widetilde{\Ric}_{[g]}$ is invertible for all metrics $g$ satisfying $\Ric(g)=cT$. If $d\,\widetilde{\Ric}_{[g]}$ is invertible, then $d\Ric_g$ has rank $\dim\M-1$, and hence $g$ is a non-degenerate critical point by statement~(2). Therefore, part (3) follows from Sard's theorem.
\end{proof}

\begin{rem}\label{Ricci_immersion}
(a)
The set of metrics $g$ with $\rank (d\Ric_g)<\dim \M-1$ is a semi-algebraic set $\mathcal S\subset\M$. As is well-known,  such a set has only finitely many components and is stratified by smooth algebraic varieties.  Since $\Ric\colon \M_T\setminus(\M_T\cap\mathcal S)\to \mathcal T  $  is an immersion,  the image of the Ricci map is thus the union of finitely many smooth connected hypersurfaces meeting along $\Ric(\mathcal S)$; see  Figures~\ref{Wallach-Ric-image} and~\ref{fig_sing_curves} for two examples. It would be interesting to know if these hypersurfaces  are embedded.

(b) In \cite{LWa} it was shown that  metrics $g$ with $\rank(d\Ric_g)=\dim \M-1$ and $S(g)\ne 0$ are  Ricci locally invertible,  which means that there exists a neighborhood $V$ of $\Ric(g)$ such that every $T\in V$ admits a metric $g'$ in a neighborhood of $g$ satisfying $\Ric(g')=cT$ for some constant~$c$. Moreover, the pair $(g',c)$ is unique up to scaling of~$g'$. It was also shown that rays through the origin in $\mathcal T$ meet the hypersurfaces $\Ric(\M\setminus\mathcal S)$ transversely. 

\end{rem}
\bigskip

\section{Saddle points on generalized Wallach spaces}\label{Wallach_saddle}
\smallskip
In this section, we prove the existence of saddle points of the functional~$S_{|\M_T}$ for a large class of Ricci candidates $T$ on generalized Wallach spaces. Recall that $G/H$ is called a \emph{generalized Wallach space} if $\fm$ splits into three irreducible modules
\begin{equation}\label{dec_Wallach}
	\fm=\fm_1\oplus\fm_2\oplus\fm_3
\end{equation}
such that $[123]$ is, up to permutations, the only non-vanishing structure constant. This implies, in particular, that we have precisely three intermediate subgroups $K_i$  with Lie algebras $\fk_i=\fh\oplus\fm_i$, where $i=1,2,3$. Furthermore, the homogeneous spaces $G/K_i$ and $K_i/H$ are all isotropy irreducible. We assume in this section that the modules $\fm_i$ are inequivalent, in particular, there exists only one decomposition $D$ up to the order of summands. There are some examples where this is not satisfied, but as we explain at the end of this section, our arguments yield results in the general case as well.

\begin{proof}[Proof of Theorem~\ref{main_2}]

Since $G/K_i$ and $K_i/H$ are isotropy irreducible,	one can easily compute the constants $\alpha_{\fk_i}$ and $\beta_{\fk_i}$ explicitly. They are given by 
\begin{align}\label{alpha_beta_Wallach}
\alpha_{\fk_i}=\frac{d_i-2[123]}{2d_iT_i} \qquad\mbox{and}\qquad \beta_{\fk_i}=\frac{d_j+d_k}{2(d_jT_j+d_kT_k)},
\end{align}
where $(i,j,k)$ is a permutation of $\{1,2,3\}$; cf.~\cite[Section~4]{AP19}.

The simplex $\Delta$ associated with the decomposition~\eqref{dec_Wallach} is two-dimensional, with vertices
$$
V_1=\Big(\frac1{d_1T_1},0,0\Big),\qquad V_2=\Big(0,\frac1{d_2T_2},0\Big) \qquad\mbox{and}\qquad V_3=\Big(0,0,\frac1{d_3T_3}\Big).
$$
These vertices are the only subalgebra strata
 $\Delta_{\fk_1}$, $\Delta_{\fk_2}$ and~$\Delta_{\fk_3}$. The metric
\begin{equation*}
g=\tfrac1{y_1}Q_{|\fm_1}+\tfrac1{y_2}Q_{|\fm_2}+\tfrac1{y_3}Q_{|\fm_3}
\end{equation*}
lies in $\M_T$ if and only if
\begin{equation*}
\tr_gT=d_1T_1y_1+d_2T_2y_2+d_3T_3y_3=1.
\end{equation*}
Its scalar curvature satisfies
\begin{align*}
S(g)=\frac{d_1}2y_1+\frac{d_2}2y_2+\frac{d_3}2y_3-\frac{[123]}2\Big(\frac{y_1y_2}{y_3}+\frac{y_2y_3}{y_1}+\frac{y_1y_3}{y_2}\Big).
\end{align*}
Without loss of generality, assume that $\beta_{\fk_i}-\alpha_{\fk_i}<0$ for $i=1,2$ and that $\alpha_{\fk_1}\ge\alpha_{\fk_2}$. One easily sees that this implies $\alpha_{\fk_2}>\alpha_{\fk_3}$ and that $\beta_{\fk_3}-\alpha_{\fk_3}>0$; see Figure~\ref{Wallach_critical} for some typical configurations.

Since $K_i/H$ are isotropy irreducible, $\alpha_{\fk_i}$ are the only critical levels at infinity. \tref{main_1} implies that the scalar curvature functional satisfies the Palais--Smale condition on $S^{-1}((a,b))\cap \M_T$ if $\alpha_{\fk_i}\notin [a,b]$ for all $i$.

We will use a mountain pass argument to prove the existence of a critical point for~$S_{|\M_T}$. For this we choose a curve $\gamma\colon\R\to\Delta$ such that $\gamma((-\infty,0])$
is the canonical variation converging to~$V_1$ and $\gamma([0,\infty))$ the canonical variation converging to~$V_2$.
These two canonical variations meet at the center $y_0=(w,w,w)$ of $\Delta$ with $w=1/(d_1T_1+d_2T_2+d_3T_3)$. Since $\beta_{\fk_i}-\alpha_{\fk_i}<0$ for $i=1,2$, one easily sees that
\begin{align*}
S(y_0)=\frac{d_1+d_2+d_3-3[123]}{2(d_1T_1+d_2T_2+d_3T_3)}>\alpha_{\fk_3}.
\end{align*}
Formula~\eqref{scal_g} implies that for a subgroup $K$ with $\beta_{\fk}-\alpha_{\fk}<0$ the canonical variation has strictly monotone scalar curvature. This means that $\inf_{t\in\R}S(\gamma(t))>\alpha_{\fk_3}$.

Let $\phi_s$ be the gradient flow of the functional $S_{|\M_T}$. Applying $\phi_s$ to $\gamma$, we obtain a family of paths $\gamma_s(t)=\phi_s(\gamma(t))$. Let
$$
c=\sup_{s\ge0}\inf_{t\in\R}S(\gamma_s(t)).
$$
Since the gradient flow increases the scalar curvature, we have
$$c\ge\inf_{t\in\R}S(\gamma(t))>\alpha_{\fk_3}.$$
We now claim that $c<\alpha_{\fk_2}$. To prove this, use Proposition~\ref{neg_der} to obtain $\e>0$ such that $S(y)<\alpha_{\fk_2}-\epsilon$ as long as $y$ lies in the difference of tubular neighbourhoods $T_{\delta_+}(\Delta_{\fk_2})\setminus T_{\delta_-}(\Delta_{\fk_2})$. We may assume that $\delta_+$ is small enough to ensure that the closure of $T_{\delta_+}(\Delta_{\fk_2})$ does not contain~$\Delta_{\fk_1}$. Along the canonical variations, we have
$$
\lim_{t\to-\infty}S(\gamma(t))=\alpha_{\fk_1}\ge\alpha_{\fk_2}\qquad\mbox{and}\qquad\lim_{t\to\infty}S(\gamma(t))=\alpha_{\fk_2}.
$$
This implies that, for sufficiently small~$\e$, the set $\{y\in\Delta\mid S(y)>\alpha_{\fk_2}-\e\}$ has at least two connected components. One of these components is contained in $T_{\delta_-}(\Delta_{\fk_2})$ and another in the complement $\Delta\setminus T_{\delta_+}(\Delta_{\fk_2})$. Each of the curves $\gamma_s$ connects them. This means that each $\gamma_s$ must pass through $T_{\delta_+}(\Delta_{\fk_2})\setminus T_{\delta_-}(\Delta_{\fk_2})$. Thus $c\le\alpha_{\fk_2}-\e$. Making $\e$ smaller if necessary, we may assume that $c\in(\alpha_{\fk_3}+\epsilon,\alpha_{\fk_2}-\epsilon)$.

Since  $S_{\M_T}$ satisfies the Palais--Smale condition on $S^{-1}((\alpha_{\fk_3}+\epsilon,\alpha_{\fk_2}-\epsilon))$, a standard argument now shows that there exists a critical point $g$  of co-index at most 1 with $S(g)=c$. Indeed, if we have no critical point at level~$c$, then the Palais--Smale condition implies that  there exist constants $\eta,\delta>0$ such that
\begin{equation*}\label{lower_bound}
|\grad S_{|\M_T}(h)|_h>\eta \qquad \text{for all} \qquad h\in S^{-1}((c-\delta,c+\delta))\cap\M_T.
\end{equation*}
But then  $S(\gamma_s(t))>c+\frac\delta2$ for all $t\in\R$ if $s$ is sufficiently large, contradicting the definition of $c$. If none of these critical points at level $c$ has co-index 0 or~1, then we can  
deform $\gamma_s$ into a curve $\tilde\gamma$ with  $\inf_{t\in\mathbb R}S(\tilde\gamma(t))>c$.
\end{proof}

It is easy to restate the conditions of \tref{main_2} in terms of the components~$T_i$ and the structure constant~$[123]$. We do this for the simplest case, the Wallach space $SU(3)/\mathbb T^2$, in \sref{sec_Examples}. Generalized Wallach spaces were classified in \cite{Ni07}. For two of them, the modules $\fm_i$ may be equivalent, namely, for $SO(n+2)/SO(n)$ and the Ledger--Obata spaces $H^4/\diag(H)$ with $H$ simple (here we have spaces of homogeneous metrics of dimensions up to six).
However, we can still use our methods to produce saddle points on these two spaces. To do so, we choose a decomposition of~$\fm$ such that a diagonal metric with respect to this decomposition has a diagonal Ricci tensor. We can do this using the symmetries induced by $N(H)/H=O(2)$, where $N(H)$ is the normalizer, and $\Aut(G,H)/\Inn(G,H)=S_4$; cf.~\cite{PZ22}. Thus critical points of the scalar curvature functional restricted to diagonal metrics in $\M_T$ are necessarily critical points of the scalar curvature functional on all of~$\M_T$. We can now argue as in the proof of Theorem B to obtain such critical points.
\smallskip

\bigskip

\section{Saddle points on  generalized flag manifolds}\label{sec_graph}
\smallskip

In this section, we prove a graph theorem for generalized flag manifolds with three or four isotropy summands. 
In the case of three summands, these spaces fall into two categories, those of type~I and type~II; see~\cite{Ki90} for a classification. One easily sees that the assumptions of \tref{main_3} are never satisfied for those  of type~II. However, these manifolds are also generalized Wallach spaces, which were discussed in~\sref{Wallach_saddle}. For them Theorem~\ref{main_2} yields critical points of $S_{|\M_T}$ of co-index 0 or~1. Notice  though that these critical points exist not below the lowest level, as in the framework of \tref{main_3}, but between intermediate critical levels at infinity.

Generalized flag manifolds with four summands also fall into two categories, those of type I and type II, as classified in~\cite{AC10}. On these spaces, unlike in the case of three summands, Maple cannot assist in finding critical points of $S_{|\M_T}$ since the Euler--Lagrange equations are too complicated.
 
We now work towards the proof of \tref{main_3}. The assumptions of Proposition~\ref{neg_der} do not hold on all generalized flag manifolds. Therefore, we need a variant of this proposition specific to our situation.

\begin{lem}\label{neg_der_2}
Consider a generalized flag manifold $G/H$ with three or four summands and an intermediate subgroup $K$ with Lie algebra $\fk$ as in \tref{main_3} and assume that $\beta_\fk-\alpha_\fk<0$. Given $\delta>0$, there exist constants $\delta_-,\delta_+,\epsilon>0$ such that $\delta_-<\delta_+\le\delta$ and $S(y)\le\alpha_{\fk}-\epsilon$ for all $y\in T_{\delta_+}(\Delta_\fk)\setminus T_{\delta_-}(\Delta_\fk)$.
\end{lem}

\begin{proof}
If $G/H$ has three summands, the result follows from \pref{neg_der} since $G/K$ is isotropy irreducible. Thus, we assume that $\fm$ splits into irreducible modules as follows:
	\begin{equation}\label{dec_flag}
		\fm=\fm_1\oplus\fm_2\oplus\fm_3\oplus\fm_4.
	\end{equation}
Let us focus on the case where $G/H$ is of type~I. Similar but simpler arguments work for type~II. Now, the only non-vanishing structure constants, up to permutations, are [112], [123], [224] and [134]. The intermediate subalgebras are
	\begin{align*}
		\fk_3=\fh\oplus\fm_3,\qquad \fk_4=\fh\oplus\fm_4\qquad\mbox{and}\qquad \fk_{24}=\fh\oplus\fm_2\oplus\fm_4.
	\end{align*}
	Figure~\ref{fig_Delta} depicts schematically the simplicial complex corresponding to $G/H$ and~$T$. We draw the subalgebra strata red and the infinity strata black.
	
	\begin{figure}[ht]
		\centering
		\includegraphics[width=50mm]{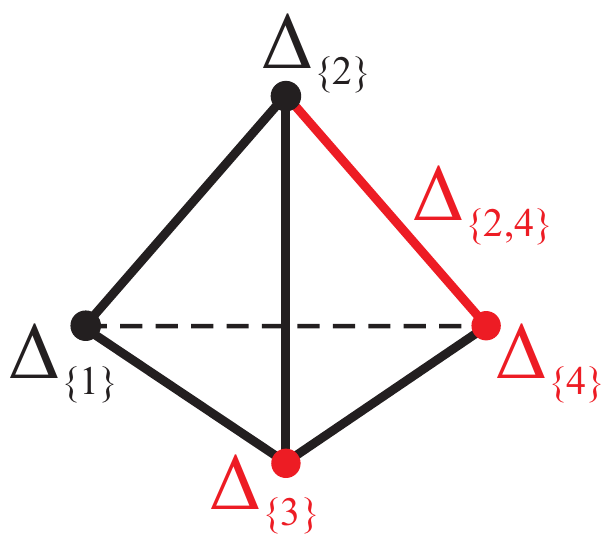}
		\caption{\centering
		The simplicial complex for a generalized flag manifold with four summands of type~I}
		\label{fig_Delta}
	\end{figure}

If $\fk=\fk_3$ or $\fk=\fk_{24}$, then $G/K$ is irreducible. In these cases, the result again follows from \pref{neg_der}. Thus, we may assume that $\fk=\fk_4$. We have
\begin{equation}\label{dec_K-inv}
\fg=\fk_4\oplus\fn_{13}\oplus\fn_2,
\end{equation}
where $\fn_{13}=\fm_1\oplus\fm_3$ and $\fn_2=\fm_2$ are $\Ad_K$-irreducible. Choose $Q$ to be the negative of the Killing form of~$G$.
As in~\eqref{est_fla1}, if $y=(y_1,\ldots,y_4)\in\Delta$, then
\begin{align*}
	S(y)
	\le\alpha_\fk \tr_{y_{|\fk\cap\fm}}T_{|\fk\cap\fm}+\frac12\sum_{i=1}^3d_iy_i&-\frac14\sum_{i,j,k=1}^3[ijk]\frac{y_iy_j}{y_k}
-\frac{y_4[413]}2\Big(\frac{y_1}{y_3}+\frac{y_3}{y_1}-2\Big).
	\end{align*}

Without loss of generality, let $y_1\le y_3$.
First, assume that
	\begin{align*}
\frac{y_3}{y_1}-1\le\Theta,\qquad \mbox{where}\qquad \Theta=\frac{\alpha_\fk-\beta_\fk}{d_3}\min_{i\in\{1,2,3\}}d_iT_i>0.
	\end{align*}
We have
\begin{align*}
	S(y)-\alpha_\fk \tr_{y_{|\fk\cap\fm}}T_{|\fk\cap\fm}&\le\frac12\sum_{i=1}^3d_iy_i-\frac14\sum_{i,j,k=1}^3[ijk]\frac{y_iy_j}{y_k}
\\	
	&\le\frac12(d_1y_1+d_3y_1+d_2y_2)-\frac{2[123]+[112]}4\Big(2y_2+\frac{y_1^2}{y_2}\Big)
\\
&\hphantom{=}~+\frac12d_3y_1\Big(\frac{y_3}{y_1}-1\Big)
+\frac{[123]}2\Big(2y_2+\frac{y_1^2}{y_2}\Big)
\\
&\hphantom{=}~-\frac{[123]}2\Big(\frac{y_1y_3}{y_2}+y_2\Big(\frac{y_1}{y_3}+\frac{y_3}{y_1}\Big)\Big).
\end{align*}
The expression in the second line of this formula equals $S(\tilde y)$, where $\tilde y$ is the $K$-invariant metric on $G/K$ given by
\begin{align*}
	\tilde y=\frac1{y_1}Q_{|\fn_{13}}+\frac1{y_2}Q_{|\fn_2}.
	\end{align*}
Clearly,
\begin{align*}
	S(\tilde y)&\le \beta_\fk\tr_{\tilde y}T_{|\fk^\perp}\le\beta_\fk\tr_{y_{|\fk^\perp}}T_{|\fk^\perp}.
	\end{align*}
Also,	
\begin{align*}
\frac{y_1y_3}{y_2}+y_2\Big(\frac{y_1}{y_3}+\frac{y_3}{y_1}\Big)\ge\frac{y_1^2}{y_2}+2y_2.
\end{align*}	
Consequently,
\begin{align*}
S(y)&\le\alpha_\fk \tr_{y_{|\fk\cap\fm}}T_{|\fk\cap\fm}+\beta_\fk\tr_{y_{|\fk^\perp}}T_{|\fk^\perp}+\frac12d_3y_1\Big(\frac{y_3}{y_1}-1\Big).
\end{align*}

Let $q\in\{1,2,3\}$ be such that $y_q=\max\{y_1,y_3,y_2\}=\max\{y_3,y_2\}$. If $y\in T_{\delta_+}(\Delta_\fk)\setminus T_{\delta_-}(\Delta_\fk)$ for some $\delta_-,\delta_+>0$ such that $\delta_-<\delta_+$, then 
$
\delta_-<y_q\le\delta_+
$
and, as in~\eqref{est_fla2},
	\begin{align*}
	S(y)
	&\le \alpha_\fk - (\alpha_\fk-\beta_\fk)\min_{i\in\{1,2,3\}}d_iT_i\Big(\delta_--\frac{\delta_+}2\Big).
	\end{align*}
Setting $\delta_-=\frac{3\delta+}4$, we obtain
	\begin{align}\label{aux_flag_negder}
	S(y)&\le \alpha_\fk - \frac{\delta_+}4(\alpha_\fk-\beta_\fk)\min_{i\in\{1,2,3\}}d_iT_i.
	\end{align}

Next, assume that
	\begin{align*}
	\frac{y_3}{y_1}-1>\Theta.
	\end{align*}
By analogy with~\eqref{est_fla3}, if $y\in T_{\delta_+}(\Delta_\fk)\setminus T_{\delta_-}(\Delta_\fk)$, then
\begin{align*}
S(y)
&\le\alpha_\fk+\frac12\sum_{i=1}^3d_iy_i-\frac12y_4[413]\Big(\frac{y_1}{y_3}+\frac{y_3}{y_1}-2\Big)
\\
&\le\alpha_\fk+\frac{\delta_+}2(d_1+d_2+d_3)-y_4\frac{[413]\Theta^2}{2(1+\Theta)}.
\end{align*}
The constraint $\tr_yT=1$ implies
\begin{align*}
y_4=\frac1{d_4T_4}\bigg(1-\sum_{i=1}^3d_iT_iy_i\bigg)\ge\frac1{d_4T_4}\bigg(1-\delta_+\sum_{i=1}^3d_iT_i\bigg).
\end{align*}
Consequently, if $\delta_+$ is sufficiently small, then
$y_4\ge\frac1{2d_4T_4}$
and
\begin{align*}
S(y)\le\alpha_\fk+\frac{\delta_+}2(d_1+d_2+d_3)-\frac{[413]\Theta^2}{4d_4T_4(1+\Theta)}
\le\alpha_\fk-\frac{[413]\Theta^2}{8d_4T_4(1+\Theta)}.
\end{align*}
Recalling~\eqref{aux_flag_negder}, we conclude that the assertion of Lemma~\ref{neg_der_2} holds with $\delta_+$ near zero, $\delta_-=\frac{3\delta_+}4$ and
\begin{align*}
\epsilon=\min\bigg\{\frac{\delta_+}4(\alpha_\fk-\beta_\fk)\min_{i\in\{1,2,3\}}d_iT_i,\frac{[413]\Theta^2}{8d_4T_4(1+\Theta)}\bigg\}.
\end{align*}
\end{proof}

\begin{proof}[Proof of Theorem~\ref{main_3}]
The simplex $\Delta$ is a triangle if $G/H$ has three irreducible isotropy summands and a tetrahedron if it has four. In either case, in light of the assumption $\alpha_\fk<\alpha_{\fk'}$ and~\eqref{inclusion}, the stratum $\Delta_\fk$ must be a vertex of~$\Delta$.  There exists at least one other vertex that is a subalgebra stratum; call it $\Delta_{\fl}$. Clearly, $\fk$ does not contain any other intermediate subalgebras, so  \lref{neg_der_2}  applies to~$\fk$.

We first observe that the scalar curvature functional $S_{|\M_T}$ satisfies the Palais--Smale condition on $S^{-1}((-\infty,a])$ for each $a<\alpha_\fk$. Indeed, this follows from Theorem~\ref{main_1} and the fact that $S_{|\M_T(K'/H)}$ has at most one critical point  if the homogeneous spaces $K'/H$ has at most two irreducible isotropy summands. This critical point is necessarily a global maximum with scalar curvature $\alpha_{\fk'}\ge\alpha_\fk$.

Now use \lref{neg_der_2} to obtain $\e>0$ such that $S(y)<\alpha_{\fk}-\epsilon$ as long as $y$ lies in the difference of tubular neighborhoods $T_{\delta_+}(\Delta_{\fk})\setminus T_{\delta_-}(\Delta_{\fk})$. We also may assume that $\delta_+$ is small enough to ensure that the closure of $T_{\delta_+}(\Delta_{\fk})$ does not contain~$\Delta_{\fl}$.
As a consequence, for sufficiently small~$\e$, the set $\{y\in\Delta\mid S(y)>\alpha_{\fk}-\e\}$ has at least two connected components. One of these components is contained in $T_{\delta_-}(\Delta_{\fk})$, and another in the complement $\Delta\setminus T_{\delta_+}(\Delta_{\fk})$.

Choose a curve $\gamma\colon\R\to\Delta$ connecting $\Delta_\fk$ to $\Delta_\fl$. Notice that $\gamma$ necessarily passes through $T_{\delta_+}(\Delta_{\fk})\setminus T_{\delta_-}(\Delta_{\fk})$. Let $\phi_s$ be the gradient flow of the functional $S_{|\M_T}$. Applying $\phi_s$ to $\gamma$, we obtain a family of curves $\gamma_s(t)=\phi_s(\gamma(t))$. Since all these curves still pass through $T_{\delta_+}(\Delta_{\fk})\setminus T_{\delta_-}(\Delta_{\fk})$, it follows that
$$
\sup_{s\ge0}\inf_{t\in\R}S(\gamma_s(t))\le\alpha_\fk-\e
$$
Standard arguments, as in the proof of Theorem~\ref{main_2}, now show that there must be a critical point of co-index $0$ or $1$ at or below the level $\alpha_\fk-\e$. This finishes the proof of Theorem~\ref{main_3}.
\end{proof}

\begin{rem}\label{rem_thm_C_hyp_relax}
The hypotheses of Theorem~\ref{main_3} can be relaxed slightly. Instead of $\alpha_\fk<\alpha_{\fk'}$, it suffices to assume that $\alpha_\fk\le\alpha_{\fk'}$ for every intermediate subalgebra~$\fk'$ and that $\fk$ has the lowest possible dimension of all the subalgebras satisfying this condition.
\end{rem}

\bigskip

\section{Examples}\label{sec_Examples}
\smallskip

In this section we consider a few examples that illustrate our theorems. We will see that  the possible behavior of critical points is richer than previously observed in the literature. In particular, we demonstrate the existence of isolated degenerate  critical points. We also examine the image of the Ricci curvature operator on homogeneous metrics.

\subsection{Generalized Wallach spaces}\label{Wallach_ex}

Here we study as a typical example the classical Wallach space $SU(3)/\mathbb T^2$. Let $G=SU(3)$ and $H=\mathbb T^2\subset SU(3)$ embedded as a maximal torus.  
On $G$ we choose as the bi-invariant metric $Q$ the negative of the Killing form $B(X,Y)=-6\tr(XY)$. With respect to the decomposition $\fm=\fm_1\oplus\fm_2\oplus\fm_3$, we have 
\begin{align}\label{metric_3dim_def}
g=x_1Q_{|\fm_1}+x_2Q_{|\fm_2}+x_3Q_{|\fm_3}\qquad\mbox{and}\qquad T=T_1Q_{|\fm_1}+T_2Q_{|\fm_2}+T_3Q_{|\fm_3}.
\end{align}
 The structure constant $[123]$ equals~$1/3$. The group $G$ has exactly three maximal connected Lie subgroups $K_1$, $K_2$ and~$K_3$ containing~$H$, namely,
\begin{equation*}
H=\mathbb T^2\subset K_i=U(2) \subset SU(3)=G,\qquad i=1,2,3,
\end{equation*}
with Lie algebras $\fk_i=\fh\oplus\fm_i$. The inclusions of $U(2)$ in $SU(3)$ are given by the three block embeddings. According to~\eqref{alpha_beta_Wallach},
\begin{align*}
	\alpha_{\fk_i}=\frac1{3T_i} \qquad\mbox{and}\qquad \beta_{\fk_i}=\frac1{T_j+T_k},
\end{align*}
where $(i,j,k)$ is a permutation of $\{1,2,3\}$.  The functional $S_{|\M_T}$ attains its global maximum if
$$
\frac{T_j+T_k}{3T_i}<1\qquad \text{with} \qquad T_i=\min\{ T_1,T_2,T_3\},
$$
 which follows from \tref{Ex_max}.
 For the existence of further critical points, \tref{main_2} implies the following.

\begin{prop}\label{prop_Wallach_maxsad}
Let $G/H$ be the Wallach space $SU(3)/\mathbb T^2$. The functional $S_{|\M_T}$ has a critical point of co-index 0 or 1 if
$$
\frac{T_j+T_k}{3T_i}>1
$$
for two distinct~$i$. 
\end{prop}
As we will see below, the co-index is in fact equal to $1$ since these critical points turn out to be non-degenerate.

The quotient $N(H)/H$, where $N(H)$ is the normalizer of $H$, is isomorphic to the permutation group on three letters. It acts on the space of metrics
by permuting~$x_i$. If $\Ric(g)=cT$, then $\Ric(n^*(g))=cn^*(T)$ for all $n\in N(H)$. Therefore, the space of Ricci candidates exhibits a natural threefold symmetry. The presence of the constant $c$ in the equations allows us to normalize~$T$. We do so by setting $T_1+T_2+T_3=1$, which preserves the threefold symmetry. Choosing the coordinates
\begin{align*}
x=\frac{4\sqrt3}3(T_1-T_2) \qquad\mbox{and}\qquad y=4T_3-\frac43
\end{align*}
in the space of normalized Ricci candidates, we mark points in the $(x,y)$-plane that correspond to different behaviors of~$S_{|\M_T}$. The large triangle in Figure~\ref{Wallach_regions} is the set of $(x,y)$ representing positive-definite~$T$. The dark-grey triangle consists of those $(x,y)$ for which \tref{Ex_max} guarantees the existence of a global maximum. In the light-grey regions, Theorem~\ref{main_2} yields a critical point of co-index 0 or~1. The red dots have a special role; see \eref{exa_curve} below. According to a Maple computation (see the discussion below), for indefinite tensors $T$ in the blue region, $S_{|\M_T}$ still has a critical point. On the other hand, for tensors $T$ in the white regions, it has no critical points. However, it may still have critical points which are indefinite metrics.

\begin{figure}[ht]
	\centering
	\includegraphics[width=74mm]{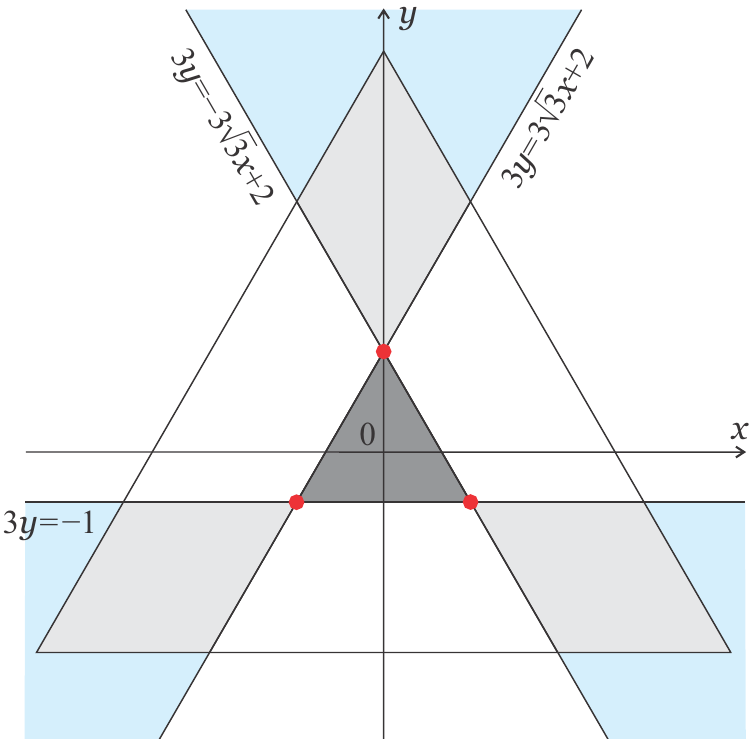}
	\caption{Regions in the $(x,y)$-plane for $SU(3)/\mathbb T^2$}
	\label{Wallach_regions}
\end{figure}
 
The scalar curvature of a metric $g$ given by~\eqref{metric_3dim_def} satisfies
\begin{eqnarray*}
	S(g)=\frac1{x_1}+\frac1{x_2}+\frac1{x_3}-\frac16\Big(\frac{x_1}{x_2x_3}+\frac{x_2}{x_1x_3}+\frac{x_3}{x_1x_2}\Big),
\end{eqnarray*}
and the components of the Ricci curvature of $g$ are given by
\begin{equation}\label{Ric_Wallach}
\ric_i=1 + \frac16\Big(\frac{x_i^2}{x_jx_k} - \frac{x_j}{x_k} - \frac{x_k}{x_j}\Big).
\end{equation}
The constraint $\tr_gT=1$ takes the form
$$
\frac{2T_1}{x_1}+\frac{2T_2}{x_2}+\frac{2T_3}{x_3}=1.
$$
 \pref{Morse} implies that $g$ is a degenerate critical point of $S_{|\M_T}$ if and only if
$
\rank(d\Ric_g)<\dim\M-1.
$
In the case of the Wallach space, $d\Ric_g\colon T_gM\to T_g\mathcal T$ can be thought of as a $3\times 3$ matrix. Setting $x_3=1$ and computing the $2\times 2$ minors of this matrix, we conclude that $\rank(d\Ric_g)<\dim\M-1$ if and only if
\begin{equation}\label{Wallach_degenerate}
x_1^4 - (2x_2^2+2)x_1^2 + x_2^4 - 2x_2^2 + 1=0.
\end{equation}
One easily sees that the solutions to this equation are given by ${x_1\pm x_2=\pm 1}$. According to~\eqref{Ric_Wallach}, the corresponding metrics have Ricci curvature $T$ with $(T_1,T_2,T_3)$ equal to $(1,1,2)$ up to permutation, depicted by the red dots in Figure~\ref{Wallach_regions}. \eref{exa_curve} below provides an explanation of this phenomenon. All the other critical points of $S_{|\M_T}$ are non-degenerate. In particular, the saddle points have co-index $1$, and all funtionals $S_{|\M_T}$ are Morse (respectively Morse-Bott) functions. Furthermore, \rref{Ricci_immersion}(a) implies that the image of the Ricci map is  the union of four smooth  hypersurfaces meeting at three points; see \fref{Wallach-Ric-image}. 

\begin{figure}[ht]
	\centering
	\includegraphics[width=80mm]{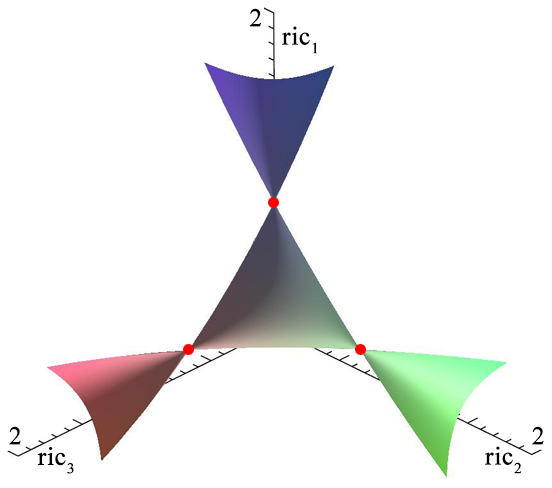}
	\caption{Image of the Ricci map on $SU(3)/\mathbb T^2$}
	\label{Wallach-Ric-image}
\end{figure}

We use Maple to demonstrate that for a tensor $T$ in the white regions in Figure~\ref{Wallach_regions} the functional $S_{|\M_T}$ has no critical points. To do so, we choose two million metrics $g$ with $0<x_i<400$. We then compute $\Ric(g)$ for such metrics and put the corresponding points $T$ into Figure~\ref{Wallach_regions}. They fill out the grey and the blue regions, with the blue regions corresponding to indefinite~$T$. This also implies that $S_{|\M_T}$ has no critical values for $T$ on the lines connecting the red dots, apart from the red dots themselves. Indeed, if a critical point were to exist, it would have to be non-degenerate and hence $S_{|\M_T}$ would have a critical point for all nearby $T$ as well. Thus global maxima and saddle points escape to infinity as $T$ approaches one of these lines.

We  can illustrate the behavior of $S_{|\M_T}$ with some explicit examples.   After choosing  values for $(T_1,T_2,T_3)$, we can solve the constraint in~\eqref{scalar_y} for~$y_1$, substitute the result into the formula for~$S(g)$, and sketch the graph of $S$ as a function of $y_2$ and~$y_3$ over the simplex $\Delta$. In Figure~\ref{Wallach_critical} one sees examples of a global maximum, a saddle point, and a case with no critical points.

\begin{figure}[ht]
	\centering
	\includegraphics[width=155mm]{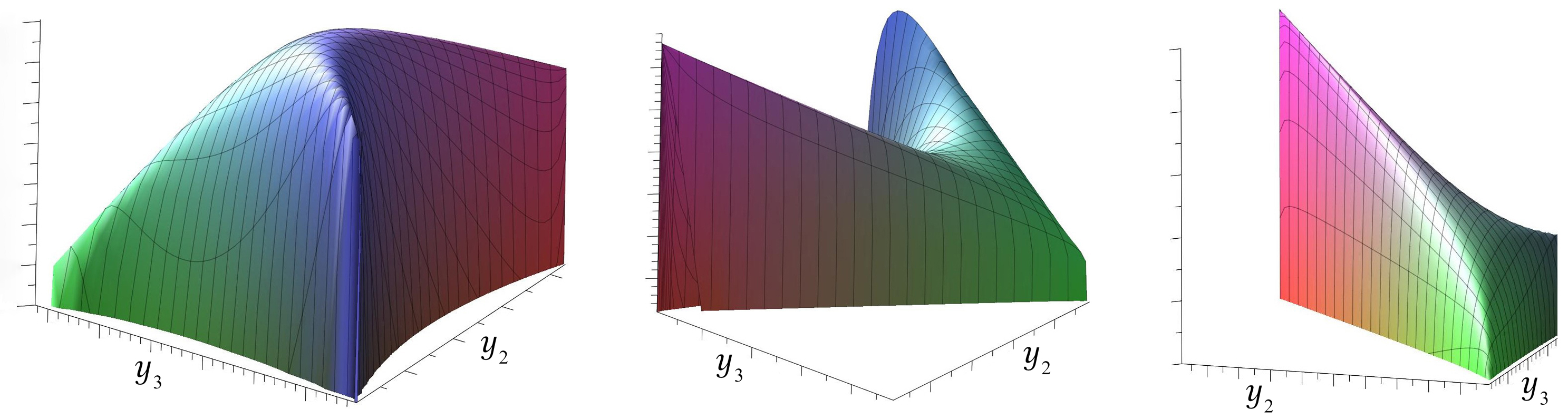}
	\caption{Critical points and absence thereof on $SU(3)/\mathbb T^2$}
	\label{Wallach_critical}
\end{figure}

As the following example shows, moving from the interior of the dark-grey triangle in Figure~\ref{Wallach_regions} to a light-grey region through a red dot, one finds that the transition is achieved by creating a curve of critical points.

\begin{example}\label{exa_curve} It is well known that on the Wallach space the K\"ahler metrics are characterized by the equations $x_k=x_i+x_j$, and they are K\"ahler--Einstein if, in addition, $x_i=x_j$. Fixing a permutation $(i,j,k)$, the corresponding K\"ahler metrics have the same Ricci curvature  (see~\cite{K55}) and hence are critical points of the corresponding functional~$S_{|\M_T}$. One easily shows that the Hessian of $S_{|\M_T}$ at each of these critical points has a zero and a negative eigenvalue. Thus these critical points form a curve, which by the Morse--Bott lemma is an isolated non-degenerate critical submanifold of~$\M_T$ and a local maximum. In Figure~\ref{Wallach_curve} we depict the graph of $S_{|\M_T}$ as a function of $y_2$ and $y_3$ with the black plane just below the critical value. This makes the curve of critical points easily visible and indicates that it is, in fact, a global maximum. The red diamond marks the K\"ahler--Einstein metric.
	\begin{figure}[ht]
		\centering
		\includegraphics[width=70mm]{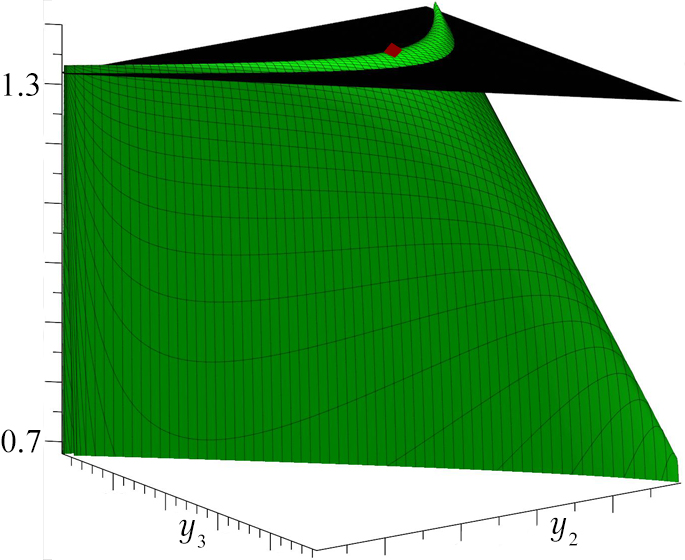}
		\caption{Curve of critical points on $SU(3)/\mathbb T^2$}
		\label{Wallach_curve}
	\end{figure}
	 \smallskip
\end{example}

Example~\ref{exa_curve} underscores the difference between $ S_{|\M_T}$ and the functional $S_{|\M_1}$ associated with the Einstein equation since critical submanifolds for $S_{|\M_1}$ are always compact; see ~\cite{BWZ04}. Another difference is that, as a critical point of $S_{|\M_1}$, the K\"ahler--Einstein metric is a saddle point of co-index 1. We also observe that the fourth Einstein metric on $SU(3)/\mathbb T^2$, corresponding to $T=Q$, is a strict local maximum of both the associated functional $S_{|\M_T}$ and the functional~$S_{|\M_1}$. 

\begin{rem}
Our analysis yields a large class of metrics that have $SU(3)$-invariant Ricci curvature even though they are not $SU(3)$-invariant themselves. More precisely, pick a point in one of the white regions in Figure~\ref{Wallach_regions}. According to our Maple experiment, the corresponding tensor field $T$ is not the Ricci curvature of any $SU(3)$-invariant metric on~$SU(3)/\mathbb T^2$. However, by a result of DeTurck's (see~\cite[Chapter~5]{AB87}), there exists a metric on a neighbourhood of every point in $SU(3)/\mathbb T^2$ with Ricci curvature equal to~$T$.
\end{rem}

\subsection{Generalized flag manifolds with three isotropy summands of type I}
\label{subsec_G2U2}
Recall that generalized flag manifolds with three summands in $\fm$ fall into two categories: those of type~I and type~II. The latter category are special examples of generalized Wallach spaces, which we discussed in the previous subsection. Thus from now on we assume that $G/H$ is a generalized flag manifold with three summands of type I. On $G$ we choose as the bi-invariant metric $Q$ the negative of the Killing form. With respect to the decomposition $\fm=\fm_1\oplus\fm_2\oplus\fm_3$, we have
\begin{align*}
g=\sum_i x_iQ_{|\fm_i}\qquad\mbox{and}\qquad T=\sum_i T_iQ_{|\fm_i},
\end{align*}
and $[112]$ and $[123]$ are the only non-vanishing structure constants, up to permutations. This implies that we have two intermediate subgroups, $K_2$ and $K_3$, with Lie algebras
\begin{equation}
\fk_2=\fh\oplus\fm_2 \qquad\mbox{and}\qquad \fk_3=\fh\oplus\fm_3.
\end{equation} 
Denote $\alpha_i=\alpha_{\fk_i}$ and $\beta_i=\beta_{\fk_i}$. The spaces $G/K_i$ and $K_i/H$ are all isotropy irreducible, and so it is easy to compute
\begin{align*}
&\alpha_2=\frac{d_2-[112]-2[123]}{2d_2T_2}, & &\alpha_3=\frac{d_3-2[123]}{2d_3T_3},\\
&\beta_2=\frac{d_1+d_3}{2d_1T_1+2d_3T_3}, & &\beta_3=\frac{2d_1+2d_2-3[112]}{4d_1T_1+4d_2T_2};
\end{align*}
cf.~\cite{AP19}.
	
For simplicity, we look at a typical case of $G_2/U(2)$, the others being similar.  Here we choose the subgroup $U(2)$ which is contained in $K_2=SO(4)$ and $K_3=SU(3)$ (corresponding to the longest root of $G_2$). Now, $d_1=4$, $d_2=2$, $d_3=4$, $[112]=2/3$ and $[123]=1/2$; see~\cite{AC11}. For a metric $g$ we have
$$
S(g)= \frac2{x_1}+ \frac1{x_2}+ \frac2{x_3} - \frac14\Big(\frac4{3 x_2} + \frac23 \frac{ x_2}{x_1^2} +\frac{x_1}{x_2x_3} + \frac{x_2}{x_1x_3}  + \frac{x_3}{x_1x_2} \Big)
$$ 
with constraint
$$
\frac{4T_1}{x_1}+	\frac{2T_2}{x_2}+	\frac{4T_3}{x_3}=1.
$$
Furthermore,
\begin{equation*}
\alpha_2=\frac{1}{12T_2},\qquad \alpha_3=\frac{3}{8T_3},\qquad \beta_2= \frac{1}{T_1+T_3}   ,\qquad \beta_3=\frac{5}{8T_1 + 4T_2}.
\end{equation*}
 
We can normalize $T$ so that $T_3=1$. Thus $\alpha_2<\alpha_3$ when $T_2>2/9$. For the ``derivatives" $\beta_2-\alpha_2$ and $\beta_3-\alpha_3$, we have
$$
\beta_2-\alpha_2<0\qquad \text{if and only if}\qquad 12T_2<T_1+1
$$
and 
$$
\beta_3-\alpha_3<0\qquad \text{if and only if}\qquad 3T_2>10-6T_1.
$$
The spaces $K_i/H$ are isotropy irreducible, and hence $S_{|\M_T(K_i/H)}$ each has exactly one critical point. \tref{main_1} implies that the limit of the scalar curvature of a divergent Palais--Smale sequence is either $\alpha_2$ or~$\alpha_3$. Away from these values (in particular, below the lower one of them), the Palais--Smale condition is satisfied. \tref{main_3} yields the following result.
	
\begin{prop}\label{G2U2_saddle}
Let $G/H$ be the homogeneous space $G_2/U(2)$. Then the functional $S_{|\M_T}$ has a critical point of co-index 0 or 1 if
\begin{equation}\label{G2U2_conditions}
\frac29 <T_2 <\frac{T_1+1}{12}\qquad
\mbox{or} \qquad \frac{10-6T_1}{3}< T_2<\frac29.
\end{equation}
\end{prop}
As we explain below, the critical points produced by this proposition, in fact, all have co-index~1. We depict the types of critical points in \fref{fig_G2_regions}. Both ``derivatives" $\beta_2-\alpha_2$ and $\beta_3-\alpha_3$ are positive in the green region  and negative in the blue region. Hence we have a global maximum and a saddle point in these regions, respectively. The most interesting case occurs in the small pink triangle in the middle with vertices $(39/25,16/75)$, $(5/3,2/9)$ and $(14/9,2/9)$. Here the higher level has a positive ``derivative" and the lower one a negative ``derivative". Thus the functional has both a global maximum and a saddle point. Maple indicates that the same holds when $T$ lies in one of the three yellow regions.

\begin{figure}[ht]
\centering
\includegraphics[width=150mm]{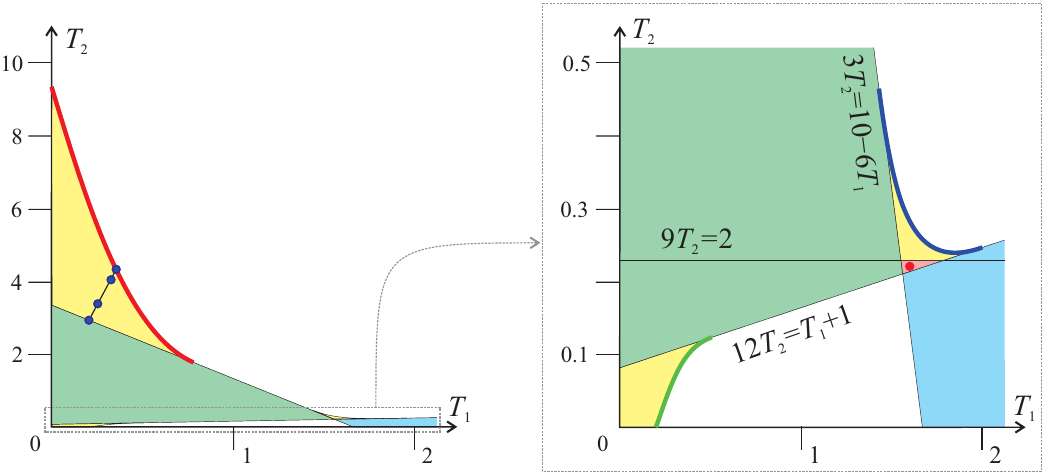}
\caption{Regions in the $(T_1,T_2)$-plane for $G_2/U(2)$}
\label{fig_G2_regions}
\end{figure}

As in the case of the Wallach space, we first determine the degenerate critical points. The three components of the Ricci tensor are 

\begin{align*}\label{Ric_G2_U(2)}
\ric_1&=\frac12-\frac{x_2}{12x_1}+\frac{x_1^2}{16x_2x_3}-\frac{x_2}{16x_3}-\frac{x_3}{16x_2},\notag
\\
\ric_2&=\frac13+\frac{x_2^2}{12x_1^2}-\frac{x_1}{8x_3}+\frac{x_2^2}{8x_1x_3}-\frac{x_3}{8x_1},\notag
\\
\ric_3&=\frac12 - \frac{x_1}{16x_2}-\frac{x_2}{16x_1}+\frac{x_3^2}{16x_1x_2}.
\end{align*}

By \pref{Morse}, a  critical point $g$ is degenerate if and only if
$\rank(d\Ric_g)<\dim\M-1$, and one easily shows that this holds only if
$$
3x_1^5 - (6x_2^2 + 6x_3^2)x_1^3 + 3(x_2 - x_3)^2(x_2 + x_3)^2x_1 - 8x_2^2x_3^3=0.
$$
Setting $x_3=1$, we find that the positive solutions to this equation are
\begin{equation}\label{G2_degenerate}
    x_1=t,\qquad x_2=\sqrt{t^2+ 1 + \frac4{3t} \pm \frac2{3t}\sqrt{9t^4 + 6t^3 + 6t + 4}}.
\end{equation}
Substituting  into the formulas
$$
T_1=\frac{\ric_1}{\ric_3}\qquad\mbox{and}\qquad T_2=\frac{\ric_2}{\ric_3},
$$
we obtain two curves, which form a portion of the boundary of the three yellow regions in Figure~\ref{fig_G2_regions}. The curve with the plus sign in \eqref{G2_degenerate} is depicted in red. The one with the minus sign brakes up into two parts, shown in bright green and dark blue. Away from these curves, all the critical points of $S_{|\M_T}$ are non-degenerate, and hence the ones in \pref{G2U2_saddle} have co-index~1.  \rref{Ricci_immersion}(a) implies that image of the Ricci map is the union of four smooth hypersurfaces depicted on the left-hand side of \fref{fig_sing_curves}. These hypersurfaces meet along curves, shown on the right-hand side, each of which projects radially onto a curve of the same color in~\fref{fig_G2_regions}.

\begin{figure}[ht]
\centering
\includegraphics[width=130mm]{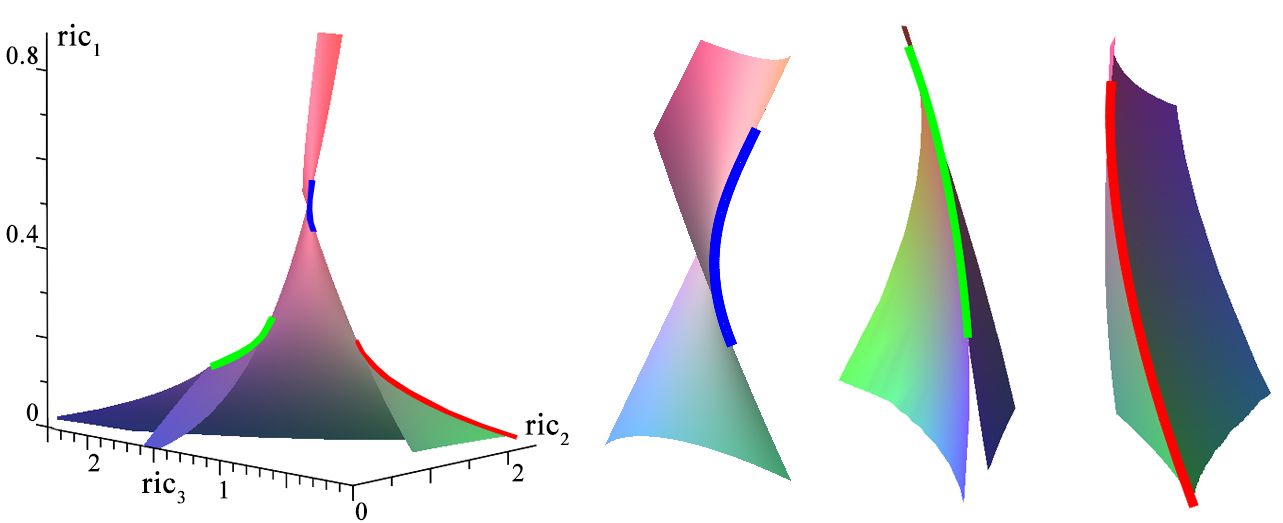}
\caption{Image and singularities of the Ricci map on $G_2/U(2)$}
\label{fig_sing_curves}
\end{figure}

We now illustrate the behavior with some explicit examples.

\begin{example}
We take a point inside the pink triangle, $(T_1,T_2)=(8/5,11/50)$, the red dot on the right-hand side of Figure~\ref{fig_G2_regions}. Our results predict the existence of a global maximum and a saddle for~$S_{|\M_T}$. This is depicted in Figure~\ref{Ex1}, where we sketch the graph of $S_{|\M_T}$ in Maple as a function of $y_1$ and $y_2$ over the simplex $\Delta$. The red and blue diamonds mark the critical points. As one can see, the existence of a global maximum and a saddle are guaranteed since the ``derivative" is positive at the highest critical level at infinity and negative at the lowest one.
\end{example}

We can exhibit the possible bifurcations of critical points as we move~$T$ along the straight line segment with four blue dots in Figure~\ref{fig_G2_regions}.

\begin{example}
We take a point on the boundary of the green region, $(T_1,T_2)=(1/5,44/15)$, the first of the blue dots. Here, the highest critical level at infinity has ``derivative" zero, and the lowest one a positive ``derivative".   Our results do not predict the existence of any critical points for~$S_{|\M_T}$. Nevertheless, sketching the graph of $S_{|\M_T}$ in Maple, we see that the functional has a unique critical point, specifically, a global maximum; see the picture on the left in \fref{fig_G2_cr_pts}.
\end{example}	

\begin{example}
We continue along our line segment to the next blue dot. The highest level has a negative ``derivative", and the lowest one a positive ``derivative". Again, our results do not predict the existence of any critical points. Nevertheless, sketching the graph in Maple, we see that the functional has a global maximum and a saddle point; see the second picture in \fref{fig_G2_cr_pts}. Notice that a curve from the global maximum to the critical level at infinity on the left ``remains hanging" on a saddle point, as one would expect from a mountain pass argument.
\end{example}

\begin{example}
We continue to the next blue dot. The global maximum becomes a local maximum, the saddle point remains, and there is no global maximum any more; see the third picture in Figure~\ref{fig_G2_cr_pts}. 
\end{example}

\begin{example}
Finally, we take the endpoint of our line segment that lies on the boundary of the yellow region. Its coordinates are
$$
(T_1,T_2)=\bigg(\frac{12\sqrt{15} - 35 }{12\sqrt{15} - 15},\frac{106\sqrt{15}}{36\sqrt{15} - 45}\bigg).
$$
Our observations above guarantee that $S_{|\M_T}$ must have a degenerate critical point. It turns out that the critical point is unique; see the graph of $S_{|\M_T}$ on the right-hand side of \fref{fig_G2_cr_pts}.
\end{example}

\begin{figure}[ht]
	\centering
	\includegraphics[width=70mm]{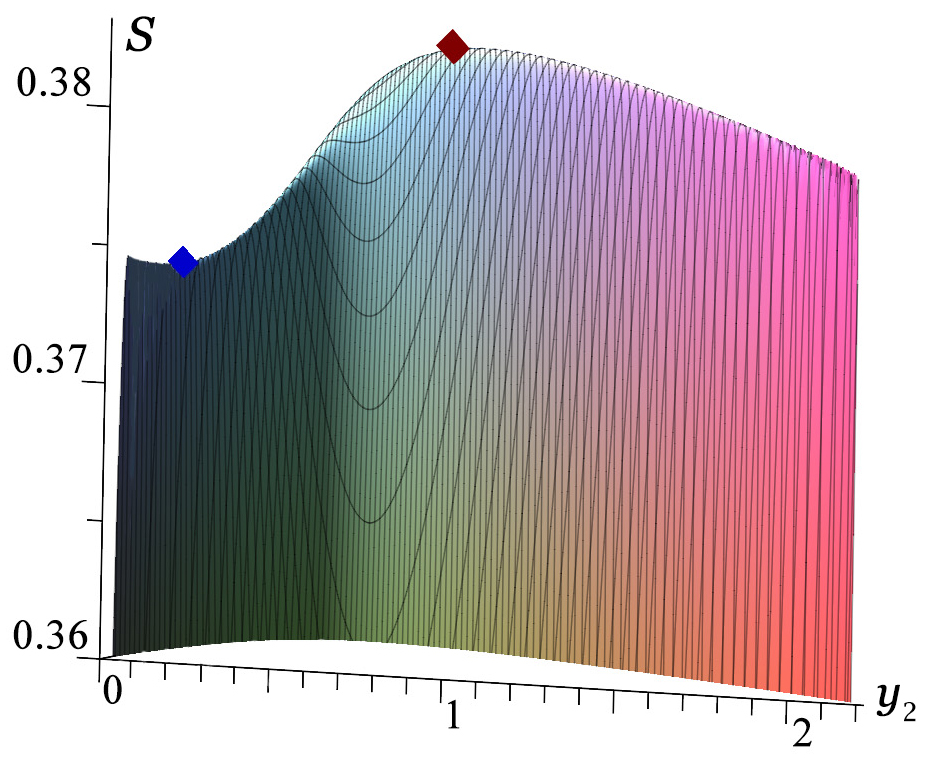}
	\caption{Critical points on $G_2/U(2)$ for a choice of $T$ in the triangle}
	\label{Ex1}
\end{figure}

\begin{figure}[ht]
\centering
\includegraphics[width=160mm]{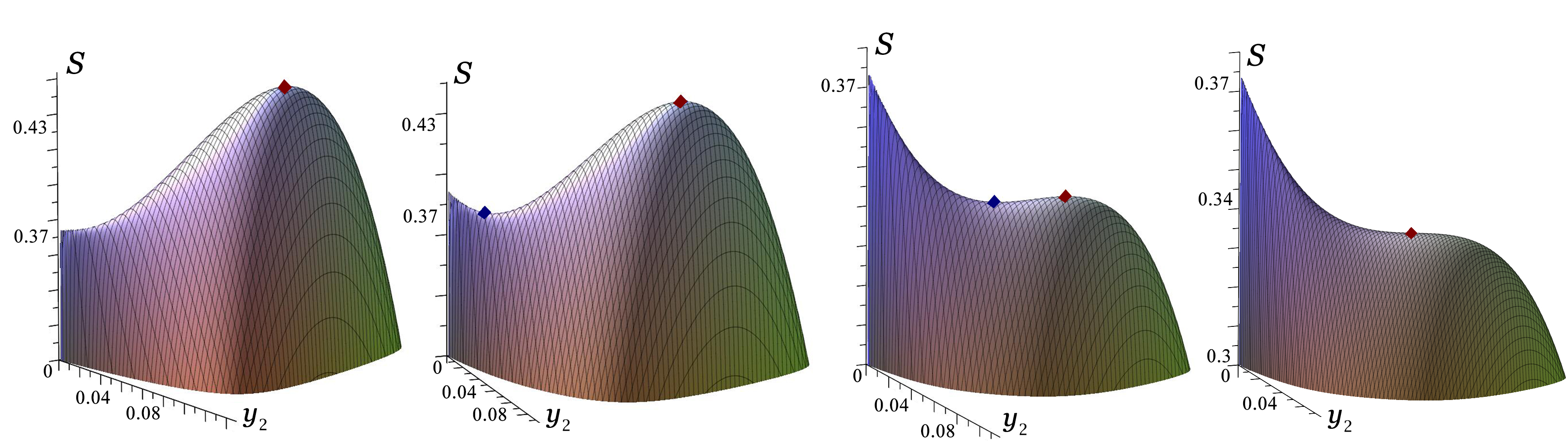}
\caption{Critical points on $G_2/U(2)$}
\label{fig_G2_cr_pts}
\end{figure}

\subsection{ Generalized flag manifolds with four isotropy summands}
\label{subsec_F4SU3SU2U1}

As in the case of three isotropy summands,  generalized flag manifolds with four isotropy summands can be of type~I or type~II; see~\cite{AC10}. We consider here an example of a generalized flag manifold of type~I. Specifically, assume that $G/H=F_4/(U(3)SU(2))$. As above, let $Q=-B$. The only non-vanishing structure constants, up to permutations, are [112], [123], [224] and~[134]. The intermediate subgroups are $K_3$, $K_4$ and $K_{24}$ with Lie algebras
\begin{align*}
\fk_3=\fh\oplus\fm_3,\qquad \fk_4=\fh\oplus\fm_4\qquad\mbox{and}\qquad \fk_{24}=\fh\oplus\fm_2\oplus\fm_4.
\end{align*}
We have
\begin{align*}
d_1=12,\qquad d_2&=18,\qquad d_3=4,\qquad d_4=6,
\\ 
[112]&=2,\qquad [123]=1,\qquad [224]=2,\qquad [134]=\frac23;
\end{align*}
see~\cite{AC10}.

Our next goal is to find the invariants $\alpha_\fk$ and $\beta_\fk$ for various choices of~$\fk$. As above, we denote $\alpha_3=\alpha_{\fk_3}$, $\alpha_4=\alpha_{\fk_4}$ and $\alpha_{24}=\alpha_{\fk_{24}}$, and similarly for the $\beta$s. Since $K_3/H$ and $K_4/H$ are isotropy irreducible, we have
\begin{align*}
\alpha_3&=
\frac{d_3-2[312]-2[314]}{2d_3T_3}=\frac1{12T_3},\qquad \alpha_4=\frac{d_4-[224]-2[314]}{2d_4T_4}=\frac2{9T_4};
\end{align*}
cf.~\cite[Eq.~(2.15)]{AP19}.
Similarly, $G/K_3$ is isotropy irreducible, and
\begin{align*}
\beta_3&=\frac{2(d_1+d_2+d_4)-3[112]-3[224]}{4(d_1T_1+d_2T_2+d_4T_4)}=\frac{5}{4T_1+6T_2+2T_4}.
\end{align*}
Let us compute~$\beta_4$. We could also compute $\alpha_{24}$ and~$\beta_{24}$, but we do not need to know their values to apply Theorem~\ref{main_3}. Indeed,~\eqref{inclusion} implies that $\fk_{24}$ can never satisfy the conditions of this theorem.

To find~$\beta_4$,  we need to maximize the scalar curvature of the metric
\begin{align*}
h=\tfrac1{y_{13}}Q_{|\fm_1\oplus\fm_3}+\tfrac1{y_2}Q_{|\fm_2}
\end{align*}
subject to the constraint
\begin{align*}
(12T_1+4T_3)y_{13}+18T_2y_2=1.
\end{align*}
Solving for $y_{13}$ and substituting into the formula for the scalar curvature, we obtain
\begin{align*}
S(h)&=8y_{13}+7y_2-\frac{y_{13}^2}{y_2}=\frac{I_1}{y_2}+I_2y_2+I_3, \qquad \mbox{where} 
\\
I_1&=-\frac1{16(3T_1+T_3)^2},\qquad
I_2=\frac{28(3T_1+T_3)^2-144(3T_1+T_3)T_2-81T_2^2}{4(3T_1+T_3)^2}
,
\\
I_3&=\frac{16(3T_1+T_3)+27T_2}{8(3T_1+T_3)^2}.
\end{align*}
This function has a critical point if and only if $I_2<0$. In this case, the critical point is a global maximum, attained at $y_2=\sqrt{\frac{I_1}{I_2}}$. Recall, however, that we need $y_2<\frac1{18T_2}$ to ensure that~$y_{13}>0$. If $I_2\ge0$ or $\sqrt{\frac{I_1}{I_2}}\ge\frac1{18T_2}$, then the supremum of $S(h)$ is the limit of $S(h)$ as $y_2\to\frac1{18T_2}$.

There are two cases to consider. One is where
\begin{align*}
7(3T_1+T_3)\ge36T_2.
\end{align*}
In this case, either $I_2\ge0$ or $\sqrt{\frac{I_1}{I_2}}\ge\frac1{18T_2}$. This means that $\beta_4=\frac{7}{18T_2}$.
The other case is where
\begin{align*}
7(3T_1+T_3)<36T_2.
\end{align*}
Then $I_2<0$, $\sqrt{\frac{I_1}{I_2}}<\frac1{18T_2}$ and
\begin{align*}
\beta_4&=\Big(\frac{I_1}{y_2}+I_2y_2+I_3\Big)_{|y_2=\sqrt{\frac{I_1}{I_2}}}=2\sqrt{I_1I_2}+I_3.
\\ &=\frac{2\sqrt{-28(3T_1+T_3)^2+144(3T_1+T_3)T_2+81T_2^2}+16(3T_1+T_3)+27T_2}{8(3T_1+T_3)^2}.
\end{align*}

Theorem~\ref{main_1} implies that the scalar curvature may converge only to $\alpha_3$, $\alpha_4$ or~$\alpha_{24}$ along divergent Palais--Smale sequences. Theorem~\ref{main_3} and Remark~\ref{rem_thm_C_hyp_relax} provide sufficient conditions for the existence of a critical point of co-index 0 or~1. We summarize these conditions in the following proposition. For convenience, we normalize $T$ so that $T_4=1$.

\begin{prop}\label{prop_4sum}
Let $G/H$ be the homogeneous space $F_4/(U(3)SU(2))$. The functional $S_{|\M_T}$ has a critical point of co-index $0$ or $1$ if the components of $T$ satisfy one of the following three collections of inequalities:
\begin{enumerate}
\item[\emph{(1)}]
$T_3\ge3/8$ and $30T_3<2T_1+3T_2+1$.
\item[\emph{(2)}]
$T_3\le3/8$, $7(3T_1+T_3)\ge36T_2$ and $T_2>7/4$.
\item[\emph{(3)}]
$T_3\le3/8$, $7(3T_1+T_3)<36T_2$ and
$$
\frac{2\sqrt{-28(3T_1+T_3)^2+144(3T_1+T_3)T_2+81T_2^2}+16(3T_1+T_3)+27T_2}{8(3T_1+T_3)^2}<\frac2{9}.
$$
\end{enumerate}
If the first collection of inequalities hold, then $\alpha_3\le\alpha_4\le\alpha_{24}$, i.e., $\alpha_3$ is the lowest level. If the second or the third collection hold, then $\alpha_4\le\alpha_3$ and $\alpha_4\le\alpha_{24}$.
\end{prop}

\begin{figure}[ht]
\centering
\includegraphics[width=160mm]{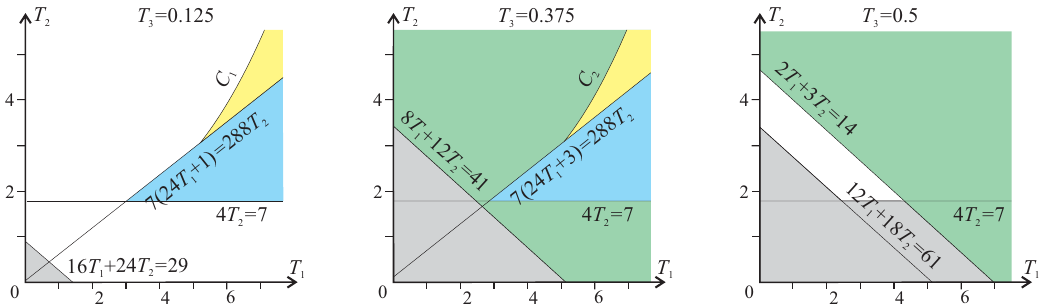}
\caption{Regions of the $(T_1,T_2)$-plane for $G_2/U(2)$}
\label{fig_F4_regions}
\end{figure}

In Figure~\ref{fig_F4_regions} we depict the sets of $T$ satisfying the conditions of the proposition. To do so, we fix $T_3$ and then draw these sets as regions in the $(T_1,T_2)$-plane. We produce pictures for a ``small" value $T_3=1/8$, the ``critical" value $T_3=3/8$, and a ``large" value $T_3=1/2$. The green regions correspond to the first collection of inequalities in Proposition~\ref{prop_4sum}, the blue regions to the second collection, and the yellow regions to the third. The curves $C_1$ and $C_2$ are given by the equation
\begin{align*}
\frac{2\sqrt{-28(3T_1+T_3)^2+144(3T_1+T_3)T_2+81T_2^2}+16(3T_1+T_3)+27T_2}{8(3T_1+T_3)^2}=\frac2{9}
\end{align*}
with $T_3=1/8$ and $T_3=3/8$, respectively. One may also use \tref{Ex_max} to find a set of $T$ for which $S_{|\M_T}$ has a global maximum; cf.~\cite[Example~5.3]{AP19}. Slices of this set are depicted in grey in \fref{fig_F4_regions}.
\bigskip

\providecommand{\bysame}{\leavevmode\hbox
	to3em{\hrulefill}\thinspace}


\begin{thebibliography}{30}
	
	\smallskip
	
	\bibitem{AC11}
S. Anastassiou, I. Chrysikos, The Ricci flow approach to homogeneous Einstein metrics on flag manifolds, J.~Geom. Phys.~61 (2011) 1587--1600.

	\bibitem{AC10} A. Arvanitoyeorgos, I. Chrysikos, Invariant Einstein metrics on flag manifolds with four isotropy summands,
	Ann. Glob. Anal. Geom. 37 (2010) 185–-219.
	
	\bibitem{AGP21}
	R.M. Arroyo, M.D. Gould, A. Pulemotov, The prescribed Ricci curvature problem for naturally reductive metrics on non-compact simple Lie groups, to appear in Comm. Anal. Geom., arXiv:2006.15765, 2020.
	
	\bibitem{APZ21}
	R.M. Arroyo, A. Pulemotov, W. Ziller, The prescribed Ricci curvature problem for naturally reductive metrics on compact Lie groups, Differential Geom. Appl. 78 (2021), article 101794.
	
    \bibitem{AB87}
    A. Besse, Einstein manifolds, Springer-Verlag, Berlin, 1987.
	
	\bibitem{CB04}
	C. B\"ohm, Homogeneous Einstein metrics and simplicial complexes, J.~Diff. Geom.~67 (2004) 79--165.

	\bibitem{BK21}
C. B\"ohm, M. Kerr, Homogeneous Einstein metrics and butterflies, Ann. Glob. Anal. Geom.~63 (2023), article~29.
	
	\bibitem{BWZ04}
	C. B\"ohm, M. Wang, W. Ziller, A variational approach for compact homogeneous Einstein manifolds, Geom. Funct. Anal. 14 (2004) 681--733.
	
	\bibitem{BK20}
	T. Buttsworth, A.M. Krishnan, Prescribing Ricci curvature on a product of spheres,  Ann. Mat. Pura Appl. 201 (2022) 1--36.
	
	\bibitem{BP19}
	T. Buttsworth, A. Pulemotov, The prescribed Ricci curvature problem for homogeneous metrics, in: Differential geometry in the large (O. Dearricott et al., eds), Cambridge University Press, 2021, 169--192.
		
	\bibitem{CD94}
	J. Cao, D.M. DeTurck, The Ricci curvature equation with rotational symmetry, Amer. J.~Math.~116 (1994) 219--241.
	
	\bibitem{D85}
	D.M. DeTurck, Prescribing positive Ricci curvature on compact manifolds, Rend. Sem. Mat. Univ. Politec. Torino~43 (1985) 357--369.
	
	\bibitem{GZ02} 
	K. Grove, W. Ziller, Cohomogeneity one manifolds with positive Ricci curvature, Invent. Math.~149 (2002) 619--646.
	
	\bibitem{RH84}
	R.S. Hamilton, The Ricci curvature equation, in: Seminar on nonlinear partial differential equations (S.-S.~Chern, ed.), Springer-Verlag, New York, 1984, 47--72.
	
		\bibitem{Ki90}
  M. Kimura,  Homogeneous Einstein Metrics 
		On Certain K\"ahler C-Spaces, in: Recent topics in differential and analytic geometry, 303–320,
		Adv. Stud. Pure Math., 18-I. 
		
		\bibitem{K55}
		J.L. Koszul, Sur la forme hermitienne canonique des espaces homog\`enes complexes, Canadian J.~Math.~7 (1955) 562--576.
		
	\bibitem{LWa}
	J. Lauret, C.E. Will, Prescribing Ricci curvature on homogeneous spaces, J. reine angew. Math. 2022 (2022) 95--133.
	
		\bibitem{LWb}
	J. Lauret, C.E. Will, On the stability of homogeneous Einstein manifolds II, J. Lond. Math. Soc. 106 (2022) 3638--3669.
	
	\bibitem{Ni07}
	Y.G. Nikonorov, Classification of generalized Wallach spaces, Geom. Dedicata 181 (2016) 193--212.

	\bibitem{AP13a}
	A. Pulemotov, Metrics with prescribed Ricci curvature near the boundary of a manifold, Math. Ann.~357 (2013) 969--986.

	\bibitem{APadd}
	A. Pulemotov, The Dirichlet problem for the prescribed Ricci curvature equation on cohomogeneity one manifolds,  
	Ann. Mat. Pura Appl.~195 (2016) 1269--1286.
	
	\bibitem{AP16}
	A. Pulemotov, Metrics with prescribed Ricci curvature on homogeneous spaces, J.~Geom. Phys.~106 (2016) 275--283.

	\bibitem{AP19}
	A. Pulemotov, Maxima of curvature functionals and the prescribed Ricci curvature problem on homogeneous spaces, J.~Geom. Anal. 30 (2020) 987--1010.
 
	\bibitem{PZ22}
	A. Pulemotov, W.Ziller, On the variational properties of the prescribed Ricci curvature functional, submitted, 	arXiv:2110.14129 [math.DG], 2021.

\bibitem{MWWZ86}
M.Y. Wang, W. Ziller, Existence and nonexistence of homogeneous Einstein metrics, Invent. Math.~84 (1986) 177--194.
\end{thebibliography}
\end{document}